\documentclass[a4paper,11pt, leqno]{article}
\usepackage[latin1]{inputenc}
\usepackage[intlimits]{amsmath}
\usepackage{amssymb}
\usepackage{amsfonts}
\usepackage{mathrsfs}
\usepackage{amsthm}
\usepackage{colonequals,comment}
\usepackage{mathrsfs}
\usepackage[all]{xy}

\newskip\procskipamount
\newskip\interskipamount
\newskip\refskipamount
\newcommand{\procskip}{\vskip\procskipamount}
\newcommand{\interskip}{\vskip\interskipamount}
\newcommand{\refskip}{\vskip\refskipamount}

\newcommand{\procbreak}{\par
   \ifdim\lastskip<\procskipamount\removelastskip
   \penalty-100
   \procskip\fi
   \noindent\ignorespaces}

\newcommand{\titlebreak}{\par%
\ifdim\lastskip<\interskipamount\removelastskip%
\penalty10000%
\interskip\fi%
\noindent}%

\newcommand{\interbreak}{\par%
\ifdim\lastskip<\interskipamount\removelastskip%
\penalty-100%
\interskip\fi%
\noindent\ignorespaces}%

\newcommand{\refbreak}{\par%
\ifdim\lastskip<\refskipamount\removelastskip%
\penalty-100%
\refskip\fi%
\noindent\ignorespaces}%

\theoremstyle{plain}
\newtheorem{theorem}{Theorem}[section]
\newtheorem{lemma}[theorem]{Lemma}
\newtheorem{corollary}[theorem]{Corollary}
\newtheorem{proposition}[theorem]{Proposition}

\newtheorem{intro-theorem}{Theorem}
\newtheorem{intro-corollary}[intro-theorem]{Corollary}
\newtheorem{intro-proposition}[intro-theorem]{Proposition}

\theoremstyle{definition}

\newtheorem{definition}[theorem]{Definition}

\newtheorem{remark}[theorem]{Remark}

\newtheorem{example}[theorem]{Example}

\numberwithin{equation}{subsection}

\newcommand{\marginrule}{\marginpar{\rule[-10.5mm]{1mm}{10mm}}}

\newcounter{refcounter}

{\par}%

\newcounter{listcounter}
\newcounter{deflistcounter}
\newcounter{equivcounter}

\newskip{\itemsepamount}
\newskip{\topsepamount}
\newenvironment{assertionlist}{%
  \begin{list}
    {\upshape (\arabic{listcounter})}
    {\setlength{\leftmargin}{18pt}
     \setlength{\rightmargin}{0pt}
     \setlength{\itemindent}{0pt}
     \setlength{\labelsep}{5pt}
     \setlength{\labelwidth}{13pt}
     \setlength{\listparindent}{\parindent}
     \setlength{\parsep}{0pt}
     \setlength{\itemsep}{\itemsepamount}
     \setlength{\topsep}{\topsepamount}
     \usecounter{listcounter}}}
  {\end{list}}

\newenvironment{definitionlist}{%
  \begin{list}
    {\upshape (\alph{deflistcounter})}
    {\setlength{\leftmargin}{18pt}
     \setlength{\rightmargin}{0pt}
     \setlength{\itemindent}{0pt}
     \setlength{\labelsep}{5pt}
     \setlength{\labelwidth}{13pt}
     \setlength{\listparindent}{\parindent}
     \setlength{\parsep}{0pt}
     \setlength{\itemsep}{\itemsepamount}
     \setlength{\topsep}{\topsepamount}
     \usecounter{deflistcounter}}}
  {\end{list}}

\newenvironment{equivlist}{%
  \begin{list}
    {\upshape (\roman{equivcounter})}
    {\setlength{\leftmargin}{18pt}
     \setlength{\rightmargin}{0pt}
     \setlength{\itemindent}{0pt}
     \setlength{\labelsep}{5pt}
     \setlength{\labelwidth}{13pt}
     \setlength{\listparindent}{\parindent}
     \setlength{\parsep}{0pt}
     \setlength{\itemsep}{\itemsepamount}
     \setlength{\topsep}{\topsepamount}
     \usecounter{equivcounter}}}
  {\end{list}}

% ...

\newenvironment{simplelist}{%
  \begin{list}{}
     {\setlength{\leftmargin}{25pt}
      \setlength{\rightmargin}{0pt}
      \setlength{\itemindent}{0pt}
      \setlength{\labelsep}{5pt}
      \setlength{\listparindent}{\parindent}
      \setlength{\parsep}{0pt}
      \setlength{\itemsep}{0pt}
      \setlength{\topsep}{0pt}}}
  {\end{list}}

\newcommand{\FF}{\mathbb{F}}
\newcommand{\GG}{\mathbb{G}}

\newcommand{\ZZ}{\mathbb{Z}}

\newcommand{\mmu}{\mu\mkern-9mu\mu}

\renewcommand{\hbar}{\bar{h}}

\newcommand{\sbar}{\bar{s}}

\newcommand{\FFbar}{\overline{\FF}}

\newcommand{\Bcal}{{\mathcal B}}

\newcommand{\Fcal}{{\mathcal F}}
\newcommand{\Gcal}{{\mathcal G}}

\newcommand{\Mcal}{{\mathcal M}}

\newcommand{\Rcal}{{\mathcal R}}
\newcommand{\Scal}{{\mathcal S}}

\newcommand{\Zcal}{{\mathcal Z}}

\newcommand{\vdot}{\dot{v}}
\newcommand{\wdot}{\dot{w}}

\newcommand{\zfr}{{\mathfrak z}}

\newcommand{\ghat}{\hat{g}}

\newcommand{\vhat}{\hat{v}}
\newcommand{\what}{\hat{w}}

\newcommand{\Ehat}{\hat{E}}

\newcommand{\Ghat}{\hat{G}}
\newcommand{\Hhat}{\hat{H}}

\newcommand{\Lhat}{\hat{L}}

\newcommand{\Phat}{\hat{P}}
\newcommand{\Qhat}{\hat{Q}}

\newcommand{\What}{\hat{W}}

\newcommand{\Hline}{\underline{H}}

\newcommand{\nline}{\underline{n}}

\newcommand{\Escr}{{\mathscr E}}
\newcommand{\Fscr}{{\mathscr F}}
\newcommand{\Gscr}{{\mathscr G}}

\newcommand{\Lscr}{{\mathscr L}}
\newcommand{\Mscr}{{\mathscr M}}

\newcommand{\Oscr}{{\mathscr O}}

\newcommand{\Lscrline}{\underline{\mathscr L}}
\newcommand{\Mscrline}{\underline{\mathscr M}}

\newcommand{\wtilde}{\tilde{w}}

\newcommand{\bgtilde}{\tilde\beta}

\DeclareMathOperator{\Aut}{Aut}

\DeclareMathOperator{\Cent}{Cent}

\DeclareMathOperator{\diag}{diag}

\DeclareMathOperator{\Gal}{Gal}
\DeclareMathOperator{\GL}{GL}
\DeclareMathOperator{\GO}{GO}
\DeclareMathOperator{\gr}{gr}

\DeclareMathOperator{\GSp}{GSp}

\DeclareMathOperator{\Hom}{Hom}

\DeclareMathOperator{\id}{id}

\DeclareMathOperator{\intaut}{int}

\DeclareMathOperator{\Lie}{Lie}

\DeclareMathOperator{\Norm}{Norm}

\newcommand{\One}{\hbox{\rm1\kern-2.3ptl}}

\DeclareMathOperator{\Par}{Par}

\DeclareMathOperator{\Pic}{Pic}

\DeclareMathOperator{\Res}{Res}
\DeclareMathOperator{\rk}{rk}
\DeclareMathOperator{\sep}{sep}

\DeclareMathOperator{\SO}{SO}

\DeclareMathOperator{\Spec}{Spec}

\DeclareMathOperator{\Sym}{Sym}

\newcommand\addots{\mathinner{\mkern1mu\raise0pt\vbox{\kern7pt\hbox{.}}\mkern2mu\raise3pt\hbox{.}\mkern2mu\raise6pt\hbox{.}\mkern1mu}}
\newcommand{\bs}{\backslash}

\newcommand{\doubleexp}[3]{\leftexp{#1}{#2}{\vphantom{#2}}^{#3}}

\newcommand{\leftexp}[2]{{\vphantom{#2}}^{#1}{#2}}

\newcommand{\lrangle}{{\langle\ ,\ \rangle}}

\newcommand{\set}[2]{\{\,#1\ ;\  #2\,\}}
\newcommand{\sett}[2]{\{\,#1\ ;\,\text{#2}\}}

\newcommand{\vdual}{{}^{\vee}}

\newcommand\lto{\longrightarrow}
\newcommand\ltoover[1]{\mathrel{\smash{\overset{#1}{\lto}}}}
\newcommand\varto[1]{\mathrel{\hbox to #1pt{\rightarrowfill}}}

\newcommand{\epi}{\twoheadrightarrow}

\newcommand{\sends}{\mapsto}

\newcommand{\iso}{\overset{\sim}{\to}}
\newcommand{\liso}{\overset{\sim}{\lto}}

\renewcommand{\implies}{\Rightarrow}
\renewcommand{\iff}{\Leftrightarrow}

\newcommand{\ksep}{{k^{\sep}}}

\newcommand{\GhatRep}{\mathop{\text{$\Ghat$-{\tt Rep}}}\nolimits}
\newcommand{\FZip}{\mathop{\text{$F$-{\tt Zip}}}\nolimits}
\newcommand{\GZipFunctor}{\mathop{\text{$G$-{\tt ZipFun}}}\nolimits}
\newcommand{\GhatZipFunctor}{\mathop{\text{$\Ghat$-{\tt ZipFun}}}\nolimits}

\begin{document}

\title{Bruhat strata and $F$-zips with additional structures}

\author{Torsten Wedhorn\footnote{Dept. of Mathematics, University of Paderborn, Warburger Str. 100, D-33098 Paderborn, Germany, {\tt wedhorn@math.uni-paderborn.de}}}

\maketitle


\noindent{\scshape Abstract.\ }
In this paper we study the Bruhat decomposition of not necessarily connected reductive quasi-split groups $G$ with respect to not necessarily connected parabolic subgroups. If $G$ is defined over a finite field, we construct a smooth morphism from the stack classifying $F$-zips with $G$-structure to the stack classifying the generalized Bruhat cells and study the relation between the resulting stratifications. We apply these general results to the twisted orthogonal $F$-zip given by the second De Rham cohomology of a relative surface, focussing on K3-surfaces.


\section*{Introduction}

\subsection*{Background}\label{background}

Let $X \to S$ be a smooth proper morphism of schemes in characteristic $p > 0$ whose Hodge spectral sequence degenerates and is compatible with base change. In~\cite{MoWd} Moonen and the author showed that its relative De Rham cohomology $H^{\bullet}_{\rm DR}(X/S)$ carries the structure of a so-called \mbox{\emph{$F$-zip}} over $S$, i.e., it is a locally free sheaf of $\Oscr_S$-modules of finite rank together with two filtrations $C^{\bullet}$ and $D_{\bullet}$ (the ``Hodge'' and the ``conjugate'' filtration) and a Frobenius linear isomorphism $\varphi_{\bullet}$ between the associated graded vector spaces (the ``Cartier isomorphism''). Very often these $F$-zips are equipped naturally with additional structures (for instance induced by the cup product, by polarizations, or by the action of an algebra on $X$). In \cite{PWZ2} Pink, Ziegler, and the author provided a general framework of $F$-zips with $\Ghat$-structure, where $\Ghat$ is a (not necessarily connected) reductive group defined over a finite field. Specializing $\Ghat$ to $\GL_n$ one obtains $F$-zips of rank $n$, specializing to other classical groups one obtains $F$-zips with certain additional structure, e.g., with a non-degenerate symmetric or alternating form.

Moreover, in \cite{PWZ2} we introduced the notion for an $F$-zips with\break \mbox{$\Ghat$-structure} to be of type $\mu$, where $\mu$ is cocharacter of $\Ghat$. The type is locally constant on the base scheme. For $F$-zips with $\GL_n$-structure the type simply determines ranks of the graded pieces of the filtration of the $F$-zips (i.e., the Hodge numbers if the $F$-zip is induced by a morphism $X \to S$ as above). $F$-zips with $\Ghat$-structure of type $\mu$ are parametrized by a quotient stack of the form $[E_{\mu} \backslash \Ghat]$, where $E_{\mu}$ is certain linear algebraic group depending on the cocharacter $\mu$. We studied this quotient stack in detail in \cite{PWZ1}, classifying the $E_\mu$-orbits in $G$ by a subset $\leftexp{\mu}{\What}$ of the Weyl group of $\Ghat$ and describing their closure relation using a variant of the Bruhat order. This reduces the study of isomorphism classes of $F$-zips with $\Ghat$-structure and their degeneration behaviour to combinatorial questions in Coxeter groups, which are nontrivial in general.

\subsection*{Bruhat strata and zip strata}

In this paper this invariant is related to a coarser invariant which encodes for $\GL_n$-zips simply the relative position of the two flags given by the filtrations $C^{\bullet}$ and $D_{\bullet}$. In general it is given by the double quotient $[\Phat\bs\Ghat/\Qhat]$, where $\Phat$ and $\Qhat$ are certain (not necessarily connected) parabolic subgroups of $\Ghat$.

Thus in the first section of the paper the quotient stack $[\Phat\bs\Ghat/\Qhat]$ is studied in detail. If $\Ghat$ is connected and split and if $\Phat = \Qhat = B$ is a Borel subgroup, then $[\Phat\bs\Ghat/\Qhat]$ parametrizes the Bruhat cells of $\Ghat$. Thus we call $[\Phat\bs\Ghat/\Qhat]$ a \emph{Bruhat stack of $\Ghat$}. We assume that $\Ghat$ is quasi-split because this simplifies the notation considerably and this assumption holds automatically in the rest of the paper, where $\Ghat$ is always defined over a finite field. If $\Ghat$ is connected and split, the quotient $[\Phat\bs\Ghat/\Qhat]$ is of course well known and the results presented here are modest generalizations to the non-connected and the non-split case. We determine the underlying topological space of the Bruhat stack (Proposition~\ref{TopBruhat}) and the dimension of all residual gerbes in each point of the Bruhat stack (Proposition~\ref{CodimBruhat}).

We then construct in the second section a smooth morphism $\beta$ from the stack of $F$-zips with $\Ghat$-structure of a fixed type $\mu$ to a certain Bruhat stack of $\Ghat$ which specializes for $\Ghat = \GL_n$ to the morphism which attaches to an $F$-zip as above the underlying flags of the two filtrations $C^{\bullet}$ and $D_{\bullet}$. Hence one gets for every $F$-zip with $\Ghat$-structure over a scheme $S$ two stratifications (i.e., decompositions into locally closed subschemes): The stratification by the isomorphism class of the $F$-zip with $\Ghat$-structure (the \emph{zip stratification}) and the stratification by its class in the Bruhat stack (the \emph{Bruhat stratification}). We describe the map induced by $\beta$ on its underlying topological spaces (Proposition~\ref{ZipBruhatTop}) and hence relate zip strata and Bruhat strata (Corollary~\ref{DescribeFibers}).

\subsection*{Applications}
\label{Appli}

In the third section we apply these results to $H^2_{\rm DR}(X/S)$, where $X/S$ is a relative surface over an $\FF_p$-scheme $S$ with $p$ odd satisfying the conditions above. We mainly focus on K3-surfaces, but abelian surfaces or Enriques surfaces yield further examples. The cup product yields the structure of an $F$-zip with $\GO(V,b)$-structure on $H^2_{\rm DR}(X/S)$, where $\GO(V,b)$ is the group of orthogonal similitudes of a split quadratic space $(V,b)$ over $\FF_p$. For $p$-principally polarized K3-surfaces we also consider the primitive De Rham cohomology. If $n := \dim(V)$ is even, then $\GO(V,b)$ is of Dynkin type $D_{n/2}$ and non-connected. If $n$ is odd, it is of Dynkin type $B_{(n-1)/2}$ and connected. We determine explicitly the underlying topological spaces of the Bruhat stack and of the stack of $F$-zips with $\GO(V,b)$-structure in this case. We show that the Bruhat stack is connected in both cases and consists of precisely three points $\id$, $s_1$, and $w_1$. In particular the primitive De Rham cohomology of a $p$-principally polarized K3-surface over a scheme $S$ yields three corresponding Bruhat strata $\leftexp{\id}{S}$, $\leftexp{s_1}{S}$, and $\leftexp{w_1}{S}$, and results of Ogus show that $\leftexp{w_1}{S}$ is the open locus where the K3-surface is ordinary and that $\leftexp{\id}{S}$ is the closed locus, where $X$ is supersingular of Artin invariant $1$.

In a further paper \cite{Wd_AStrat} we will study the Bruhat stratification for Shimura varieties of PEL type and use the results obtained here to prove geometric properties of these strata such as smoothness and their dimension.


\section{The Bruhat stack}\label{Bruhat}


\subsection{Double quotient stacks}\label{QuotStacks}
Let $X$ be a set, let $G$ be a group acting on $X$ from the left. We denote by $[G \backslash X]$ the following category. Objects are the elements of $X$ and for $x,x' \in X$ we set
\[
\Hom_{[G \backslash X]}(x,x') := \set{g \in G}{gx = x'}.
\]
Composition is defined by multiplication in the group $G$. Then this category is a groupoid (i.e., every morphism is an isomorphism), and the isomorphism classes in $[G \backslash X]$ are in bijective correspondence to the set of $G$-orbits on $X$.

Now let $k$ be a field, $S = \Spec k$, let $Z$ be a $k$-scheme, let $G$ be a group scheme of finite type over $k$ acting on $Z$ from the left. For every affine $k$-scheme $U = \Spec R$ let $[G \backslash Z]'_U$ be the groupoid $[G(R) \backslash Z(R)]$. Every morphism of affine $k$-schemes $f\colon U_1 \to U_2$ induces by functoriality a functor $f^*\colon [G \backslash Z]'_{U_2} \to [G \backslash Z]'_{U_1}$. It is easy to check that we obtain a $k$-groupoid $[G \backslash Z]'$. Let $[G \backslash Z]$ its stackification. Then it is shown in~\cite{Stacks}~Tag~04WM that $[G \backslash Z]$ can be identified with the quotient stack defined in \cite{LM}~(2.4.2). It is an algebraic stack by \cite{LM}~(10.13.1). 

Similarly we define an algebraic stack $[Z/G]$ if $G$ is acting from the right on $Z$.

Let $H$ be a subgroup scheme of $G$. Then $[H \bs G]$ is representable by a scheme of finite type which we denote simply by $H \bs G$.

Let $K$ be a second subgroup scheme of $G$. Let $[K\bs_{-1} G]$ be the quotient by the left action $K \times G \to G$, $(k,g) \sends gk^{-1}$. Then $\id_G$ on objects and $k \sends k^{-1}$ on morphisms yields an isomorphism of algebraic stacks (in fact of schemes) $[K \bs_{-1} G] \iso [G/K]$ which we use to identify these two stacks.

We let $H \times K$ act on $G$ from the left by $(h,k)\cdot g:= hgk^{-1}$ and denote the corresponding quotient stack simply $[H\backslash G/K]$. For $K = 1$ or~$H = 1$ we obtain the schemes $H\backslash G$ and $G/K$, respectively. Multiplication from the left defines a left action from $H$ on $G/K$ and multiplication from the right defines a right action from $K$ on $H\backslash G$. It is easy to check (on the level of $k$-groupoids $[\ ]'$ as above) that there are equivalences
\begin{equation}\label{EquivalDouble}
[H\backslash G/K] \cong [H \backslash (G/K)] \cong [(H\backslash G)/K].
\end{equation}


We also have the following lemma.

\begin{lemma}\label{DoubleQuotient}
The morphism
\[
[H \backslash G/K] \lto [G \backslash (G/H \times_k G/K)]
\]
induced by $g \sends (1,g)$ is an isomorphism of algebraic stacks.
\end{lemma}

\begin{proof}
It is straight forward to check that $g \sends (1,g)$ already yields an isomorphism of $k$-groupoids $[K \backslash (G/H)]' \lto [G \backslash (G/H \times_k G/K)]'$. Its inverse is given by $(g_1,g_2) \sends g_1^{-1}g_2$. In particular its stackification is an isomorphism.
\end{proof}


\subsection{General notation}\label{GenNot}

Let $k$ be a field, choose a separable closure $\ksep$ of $k$, and let $\Gal(\ksep/k)$ be its Galois group. For every $k$-scheme $X$ and every $k$-algebra $R$, we set $X_R := X \otimes_k R$.

By a linear algebraic group over $k$ we mean a smooth affine group scheme over~$k$. If $\hat H$ is a linear algebraic group over $k$, we denote its identity component by~$H$ and the finite \'etale group scheme of connected components by $\pi_0(\Hhat) := \Hhat/H$; and similarly for other letters of the alphabet. Note that the unipotent radical $R_uH$ of $H$ (if it exists over $k$) is a normal subgroup of~$\Hhat$. Any homomorphism of algebraic groups $\hat\phi\colon\Ghat\to\Hhat$ restricts to a homomorphism $\phi\colon G\to H$.

From now on let $\Ghat$ be a linear algebraic group over $k$ such that $G$ is reductive. To simplify the assertions we assume that $G$ is quasi-split (this is the only case we will use). We choose a maximal torus $T$ and a Borel subgroup $B \supseteq T$ of $G$. Consider the finite groups
\begin{eqnarray*}
     W &:=& \Norm_{ G   (\ksep)}(T(\ksep))/T(\ksep), \\
\hat W &:=& \Norm_{\Ghat(\ksep)}(T(\ksep))/T(\ksep), \\
\Omega &:=& \bigl(\Norm_{\Ghat(\ksep)}(T(\ksep)) \cap \Norm_{\Ghat(\ksep)}(B(\ksep))\bigr)\!\bigm/\!T(\ksep).
\end{eqnarray*}
The fact that $W$ acts simply transitively on the set of Borel subgroups containing $T_\ksep$ implies that $\hat W = W \rtimes \Omega$, and the fact that $G(\ksep)$ acts transitively on the set of all maximal tori of $G_\ksep$ implies that $\Omega \cong \hat W/W \cong \pi_0(\hat G)(\ksep)$.
Also, let $I \subset W$ be the set of simple reflections associated to the pair $(T,B)$. As this pair is unique up to conjugation by $G(\ksep)$ and as $\Norm_{G(\ksep)}(T(\ksep)) \cap \Norm_{G(\ksep)}(B(\ksep)) = T(\ksep)$, the Coxeter system $(W,I)$ and the groups $\hat W$ and $\Omega$ are, up to unique isomorphism, independent of the choice of $T$ and~$B$. For any $w \in W$ we fix a representative $\dot w \in \Norm_G(T)(\ksep)$. By choosing representatives attached to a Chevalley system (see \cite{SGA3}~Exp.~XXIII, \S6) for all $w_1, w_2 \in W$ with $\ell(w_1w_2) = \ell(w_1) + \ell(w_2)$ we obtain
\begin{equation}\label{DotMultEq}
\dot w_1 \dot w_2 = (w_1w_2)\dot{\ }.
\end{equation}
In particular the identity element $1\in W$ is represented by the identity element $1\in G(\ksep)$.

The Galois group $\Gal(\ksep/k)$ acts continuously on the (discrete) group $\What$ preserving $W$, $\Omega$ and $I$. In particular it acts on $(W,I)$ by automorphisms of Coxeter systems. For any family $X = (X_i)_{i}$ of subsets of $\What$ we call its \emph{field of definition} the finite field extension $\kappa(X) \supseteq k$ in $\ksep$ given by
\[
\Gal(\ksep/\kappa(X)) = \set{\gamma \in \Gal(\ksep/k)}{\forall\,i: \gamma(X_i) = X_i}.
\]

\begin{lemma}\label{WdotRational}
Assume that $G$ is semisimple adjoint or that $G$ is semisimple simply connected or that the cohomological dimension of $k$ ist $\leq1$. Then for $w \in W$ the representative $\wdot$ can be chosen in $\Norm_G(T)(\kappa(w))$.
\end{lemma}

\begin{proof}
Considering $W$ as a finite \'etale group scheme over $k$, the fiber of $\Norm_G(T) \to W$ over the point $\Spec \kappa(w) \to W$ is a torsor under $T$. But $H^1(k,T) = 0$ under any of these assumptions: If ${\rm cd}(k) \leq 1$, then this is Lang's theorem (e.g.\ \cite{Serre_GalCoh}). As $T$ is contained in a Borel subgroup of $G$, the other assumptions imply that $T$ is the product of tori of the form $\Res_{l/k}(\GG_m)$, where $l \supseteq k$ is a finite separable extension (\cite{SGA3}~Exp.~XXIV, Prop.~3.13) which also implies $H^1(k,T) = 0$ by Shapiro's lemma and Hilbert 90.
\end{proof}

Recall that the length of an element $w\in W$ is the smallest number $\ell(w)$ such that $w$ can be written as a product of $\ell(w)$ simple reflections and that the Bruhat order $\leq$ on $W$ is defined by $w'\leq w$ if for some (and equivalently for any) expression of $w$ as a product of $\ell(w)$ simple reflections, by leaving out certain factors one can obtain an expression of $w'$ as a product of $\ell(w')$ simple reflections. As the action of $\Gal(\ksep/k)$ on $W$ preserves $I$ it also preserves the length of elements and the Bruhat order.

We extend the Bruhat order to $\What = W \rtimes \Omega$ by defining
\begin{equation}\label{BruhatOrderExtended}
w\omega \leq w'\omega' \quad :\iff \quad \text{$w \leq w'$ and $\omega = \omega'$}
\end{equation}
for $w,w' \in W$ and $\omega,\omega' \in \Omega$.

For any subsets $J,K \subseteq I$, we denote by $W_J$ the subgroup of $W$ generated by $J$ and by $\leftexp{J}{W}$ (resp.\ $W^{K}$, resp.\ $\doubleexp{J}{W}{K}$) the set of $w \in W$ that are of minimal length in the left coset $W_Jw$ (resp.\ in the right coset $wW_{K}$, resp.\ in the double coset $W_JwW_{K}$). Then $\doubleexp{J}{W}{K} = \leftexp{J}{W} \cap W^K$.

We let $w_0 \in W$ denote the unique element of maximal length in $W$, and $w_{0,J}$ the unique element of maximal length in~$W_J$. Then
\begin{equation}\label{MaxLengthElt}
w_{0,J}w_0 \in \leftexp{J}{W}
\end{equation}
is the unique element of maximal length in $\leftexp{J}{W}$ and
\begin{equation}\label{MaxLength}
\ell(w_{0,J}w_0) = \ell(w_0) - \ell(w_{0,J}).
\end{equation}


\subsection{Automorphisms of the Weyl group induced by isogenies}\label{IsogAut}

Let $\varphi\colon G_{\ksep} \to G_{\ksep}$ be an isogeny, i.e. a finite surjective homomorphism of linear algebraic groups. Then $\varphi(T_{\ksep})$ is a maximal torus of $G_{\ksep}$ and $\varphi(B_{\ksep})$ is a Borel subgroup containing $\varphi(T_{\ksep})$.
Let $g \in G(\ksep)$ such that $\leftexp{g}{\varphi(B_{\ksep})} = B_{\ksep}$ and $\leftexp{g}{\varphi(T_{\ksep})} = T_{\ksep}$. As $\Norm_G(B) \cap \Norm_G(T) = T$, the element $g$ is unique up to left multiplication with an element of $T(\ksep)$. Therefore the isomorphism of Coxeter systems
\begin{equation}\label{IsogenyCoxeter}
\bar\varphi\colon (W,I) \iso (W,I)
\end{equation}
induced by ${\rm int}(g) \circ \varphi\colon \Norm_G(T)(\ksep) \to \Norm_G(T)(\ksep)$ does not depend on the choice of $g$.

Let $\psi\colon G_{\ksep} \to G_{\ksep}$ be a second isogeny. We claim that
\begin{equation}\label{CompIsog}
\overline{\varphi \circ \psi} = \bar\varphi \circ \bar\psi.
\end{equation}
Indeed, let $h \in G(\ksep)$ be such that
\[
\leftexp{h}{(\psi(B_{\ksep}), \psi(T_{\ksep}))} = (T_{\ksep},B_{\ksep}).
\]
Then
\[
\leftexp{g\varphi(h)}{(\varphi(\psi(B_{\ksep})), \varphi(\psi(T_{\ksep})))} = (T_{\ksep},B_{\ksep}).
\]
This shows the claim.

For $\ghat \in \Ghat(\ksep)$, conjugation with $\ghat$ yields an automorphism of $G_{\ksep}$ and hence an automorphism $\delta(\ghat)$ of $(W,I)$ as in \eqref{IsogenyCoxeter}. By \eqref{CompIsog}, we obtain a homomorphism of groups $\delta\colon \Ghat(\ksep) \to \Aut(W,I)$. For $g \in G(\ksep)$ one has $\delta(g) = 1$ by definition and hence one obtains a homomorphism
\begin{equation}\label{Pi0Aut}
\delta\colon \pi_0(G)(\ksep) \to \Aut(W,I).
\end{equation}
It is easy to check that this action is the canonical action of $\Omega$ on $W$ via the isomorphism $\pi_0(G)(\ksep) \cong \Omega$.


\subsection{Parabolic subgroups of non-connected reductive groups}\label{NonConnPSGP}

We fix a subset $J$ of $I$ and set $\kappa = \kappa(J)$. We denote by $P^{\sep}_J$ the unique parabolic subgroup of the (split) reductive group $G_{\ksep}$ of type $J$ which contains $B_{\ksep}$. As $G$ is quasi-split over $k$, we have the following result (cf.\ \cite{BoTi}~6.3).

\begin{lemma}
The field of definition of $P^{\sep}_J$ is $\kappa(J)$.
\end{lemma}

In particular there is a (necessarily unique) parabolic subgroup $P_J$ of $G_{\kappa(J)}$ of type $J$ containing $B_{\kappa(J)}$. 

As $\Norm_{\Ghat_{\kappa}}(P_J) \cap G_{\kappa} = P_J$, the canonical homomorphism
\[
\Norm_{\Ghat_{\kappa}}(P_J)/P_J \to \pi_0(\Ghat_{\kappa})
\]
is an open and closed immersion. Its image $\Omega_J$ is the finite \'etale group scheme such that $\Omega_J(\ksep)$ is the finite subgroup of $\pi_0(G)(k^{\rm sep})$ that consists of those connected components $G'$ of $G_{\ksep}$ such that for one (or, equivalently, for all) $g' \in G'(\ksep)$ the parabolic subgroup $\leftexp{g'}{(P_J)_{\ksep}}$ of $G_{\ksep}$ is $G(\ksep)$-conjugate to $(P_J)_{\ksep}$. In other words, $\Omega_J$ is the stabilizer of $J$ under the action $\delta$~\eqref{Pi0Aut}.

We fix a closed \'etale $\kappa$-subgroup scheme $\Theta$ of $\Norm_{\Ghat_{\kappa}}(P_J)/P_J$ and define an algebraic subgroup $\hat P_{J,\Theta}$ of $\Ghat_{\kappa}$ as the inverse image of $\Theta$ under the composition
\[
\Norm_{\Ghat_{\kappa}}(P_J) \to \Norm_{\Ghat_{\kappa}}(P_J)/P_J \cong \Omega_J \subset \pi_0(\Ghat_{\kappa}).
\]
The identity component of $\Phat_{J,\Theta}$ is indeed $P_J = \Phat_{J,\Theta} \cap G_{\kappa}$ and $\pi_0(\Phat_J) = \Theta$.

We denote by
\[
\Par_{J,\Theta} := \Par_{\Ghat,J,\Theta} := \Ghat_{\kappa}/\Phat_{J,\Theta}
\]
the quotient. Again this quotient depends on the choice of $T$ and $B$ only up to unique isomorphism. If $\Ghat = G$ is connected, one necessarily has $\Theta = 1$ and
\[
\Par_J := \Par_{G,J,1} = G_{\kappa}/P_J
\]
is the standard partial flag variety of parabolic subgroups of $G$ of type $J$.


\subsection{Bruhat decomposition for not necessarily connected \break groups}\label{BruhatDecomp}

In this subsection we assume that $k = \ksep$ is separably closed. All schemes $X$ occuring in this section are smooth over $k$, allowing us to confuse $X$ and $X(k)$. We fix two subsets $J,K \subseteq I$ and two subgroups $\Theta \subseteq \Omega_J$ and $\Delta \subseteq \Omega_K$. We want to describe the $(\Phat_{J,\Theta},\Phat_{K,\Delta})$-double cosets of $\Ghat$. To do this we recall first the connected case and then generalize to the non-connected case.

Let $T \subset B$ be a Borel pair of $G$. Let $N = \Norm_G(T)$. If $\Ghat = G$ is connected, it is well known (e.g., \cite{Tits_Building} Theorem~5.2), that as $G$ is split over $\ksep$, the tuple $(G(\ksep), B(\ksep), N(\ksep), I)$ is a Tits system in the sense of~\cite{Bou_Lie}~Chap.~IV, \S2.1, Def.~1. The subset $\doubleexp{J}{W}{K}$ is a system of representatives for the set of double quotients $W_J\backslash W/W_K$. Hence the Bruhat decomposition for Tits system (\cite{Bou_Lie}~Chap.~IV, \S2.5, Rem.~2) implies that the map
\begin{equation}\label{BruhatConnected}
\begin{aligned}
\doubleexp{J}{W}{K} &\to P_J\backslash G/P_K, \\
w &\sends C_{J,K}(w) := P_J\wdot P_K
\end{aligned}
\end{equation}
is bijective (and independent of the choice of the representative $\wdot$ in $N$ for $w \in W$). Moreover, by \cite{BoTiAdd}~\S3 one has for $w,w' \in \doubleexp{J}{W}{K}$
\begin{equation}\label{ClosureBruhat}
C_{J,K}(w') \subseteq \overline{C_{J,K}(w)} \iff w' \leq w.
\end{equation}

We will now generalize this first to the case that $\Ghat$ is not necessarily connected, but still $\Phat_{J,\Theta} = P_J$ and $\Phat_{K,\Delta} = P_K$. For each $\omega \in \Omega$ fix a representative $\dot\omega \in \Norm_{\Ghat}(B) \cap \Norm_{\Ghat}(T)$ and for $\hat{w} = w\omega \in \What$ with $w \in W$ and $\omega \in \Omega$ set $\dot{\what} := \wdot\dot{\omega}$. 

Define
\begin{equation}\label{JWHatK}
\doubleexp{J}{\What}{K} := \set{\what = w\omega \in \What = W\Omega}{\omega \in \Omega, w \in \doubleexp{J}{W}{\omega(K)}}.
\end{equation}

\begin{lemma}\label{NCBruhat1}
The map
\[
\doubleexp{J}{\What}{K} \to P_J\backslash \Ghat/P_K,\qquad
\what \sends P_J\dot{\what}P_K
\]
is a bijection. One has
\[
P_J\dot{\what}'P_K \subseteq \overline{P_J\dot{\what}P_K} \iff \what' \leq \what,
\]
where ``$\leq$'' denotes the extended Bruhat order \eqref{BruhatOrderExtended}.
\end{lemma}

\begin{proof}
The subgroup $\leftexp{\dot{\omega}}{P_K}$ is a parabolic subgroup of $G$ of type $\omega{K}$. As $\dot\omega \in \Norm_{\Ghat}(B)$, it contains $B$ and hence $\leftexp{\dot{\omega}}{P_K} = P_{\omega(K)}$. Therefore one has $P_J\dot{w}\dot\omega P_K = P_J\wdot P_{\omega(K)}\dot\omega$. As $\Ghat = \coprod_{\omega \in \Omega}G\dot\omega$, the lemma follows from the Bruhat decomposition in the connected case (\eqref{BruhatConnected} and~\eqref{ClosureBruhat}).
\end{proof}

Finally consider the general case. Let $L_J$ be the unique Levi subgroup of $P_J$ containing $T$ and set
\begin{equation}\label{DefineLhatJ}
\Lhat_{J,\Theta} := \Norm_{\Phat_{J,\Theta}}(L_J).
\end{equation}
Define subgroups of $\What$ as follows
\begin{equation}\label{Define_JTheta}
\begin{aligned}
\What_{J,\Theta} &:= \Norm_{\Lhat_{J,\Theta}}(T)/T, \\
\Omega_{J,\Theta} &:= (\Norm_{\Lhat_{J,\Theta}}(B) \cap \Norm_{\Lhat_{J,\Theta}}(T))/T.
\end{aligned}
\end{equation}
These subgroups satisfy
\begin{equation}\label{Prop_JTheta}
\begin{aligned}
\What_{J,\Theta} &= W_J \rtimes \Omega_{J,\Theta}, \\
\Omega_{J,\Theta} &\cong \pi_0((\Lhat_{J,\Theta})_{\ksep}) \cong \pi_0((\Phat_{J,\Theta})_{\ksep}).
\end{aligned}
\end{equation}
As $\dot\omega \in \Norm_{\Ghat}(P_J)$ we have
\begin{equation}\label{CompPhatJ}
\Phat_{J,\Theta} = \coprod_{\omega \in \Omega_{J,\Theta}}P_J\dot\omega = \coprod_{\omega \in \Omega_{J,\Theta}}\dot\omega P_J.
\end{equation}
We have analogous definitions and properties for $L_K$, $\Lhat_{K,\Delta}$, $\What_{K,\Delta}$, and $\Omega_{K,\Delta}$.

\begin{lemma}
Let $\omega_1 \in \Omega_J$ and $\omega_2 \in \Omega_K$.
\begin{assertionlist}
\item
For $\what \in \doubleexp{J}{\What}{K}$ one has again $\omega_1\what\omega_2 \in \doubleexp{J}{\What}{K}$.
\item
Let $\what, \what' \in \What$ with $\what \leq \what'$. Then $\omega_1\what\omega_2 \leq \omega_1\what'\omega_2$.
\end{assertionlist}
\end{lemma}

\begin{proof}
As explained in the definition of $\Omega_J$, we have $\omega_1(J) = J$ and $\omega_2(K) = K$. Let $\what = w\omega$ with $\omega \in \Omega$ and $w \in \doubleexp{J}{W}{\omega(K)}$. Then $\omega_1w\omega \omega_2 = \omega_1(w) \omega_1\omega\omega_2$ with
\[
\omega_1(w) \in \doubleexp{\omega_1(J)}{W}{\omega_1\omega(K)} = \doubleexp{J}{W}{\omega_1\omega\omega_2(K)}
\]
and hence $\omega_1w\omega \omega_2 \in \doubleexp{J}{\What}{K}$.

The second assertion follows from the fact that the action of $\Omega$ on $W$ preserves the set of simple reflections and hence the Bruhat order.
\end{proof} 

The Lemma shows that we can form the double quotient
\begin{equation}\label{WhatGeneral}
\doubleexp{J,\Theta}{\What}{K,\Delta} := \Omega_{J,\Theta}\backslash \doubleexp{J}{\What}{K}/\Omega_{K,\Delta}
\end{equation}
and that the partial order on $\What$ induces a partial order on $\doubleexp{J,\Theta}{\What}{K,\Delta}$. For $\what \in \doubleexp{J}{\What}{K}$ we denote by $[\what] \in \doubleexp{J,\Theta}{\What}{K,\Delta}$ its image. Then we deduce from Lemma~\ref{NCBruhat1} and from~\eqref{CompPhatJ}:

\begin{proposition}\label{NCBruhat}
The map
\[
\doubleexp{J,\Theta}{\What}{K,\Delta} \to \Phat_{J,\Theta}\backslash \Ghat/\Phat_{K,\Delta},\qquad
[\what] \sends
\Phat_{J,\Theta}\dot{\what}\Phat_{K,\Delta}
\]
is a bijection. One has
\[
\Phat_{J,\Theta}\dot{\what}'\Phat_{K,\Delta} \subseteq \overline{\Phat_{J,\Theta}\dot{\what}\Phat_{K,\Delta}} \iff [\what'] \leq [\what].
\]
\end{proposition}

\begin{remark}\label{WConnComp}
Recall that we may consider any partially ordered set $(Y,\leq)$ as a topological space (a subset $U \subset Y$ is open if and only if for all $u \in U$ one has $\set{y \in Y}{u \leq y} \subseteq U$). In particular we may consider $\doubleexp{J,\Theta}{\What}{K,\Delta}$ as a topological space. For elements $\what = w\omega$ and $\what' = w'\omega'$ in $\What$ we have by definition $\what \leq \what'$ if and only if $w \leq w'$ and $\omega = \omega'$. Thus $\What \to \Omega$, $w\omega \sends \omega$ induces a continuous map
\begin{equation}\label{EqWConnComp}
\doubleexp{J,\Theta}{\What}{K,\Delta} \to \Omega_{J,\Theta}\backslash \Omega/\Omega_{K,\Delta},
\end{equation}
where the right hand side is endowed with the discrete topology (associated to the trivial partial order).
\end{remark}

The topological space associated to a partially ordered set of the form $\doubleexp{J}{W}{K}$ is irreducible (and in particular connected) because the unique maximal element of $\doubleexp{J}{W}{K}$ is a generic point. In particular we obtain the following result.

\begin{lemma}\label{ConnCompW}
The fibers of~\eqref{EqWConnComp} are the connected components of $\doubleexp{J,\Theta}{\What}{K,\Delta}$.
\end{lemma}


\subsection{The Bruhat stack}

Now let $k$ be again an arbitrary field. We keep fixing two subsets $J,K \subseteq I$ and subsets $\Theta \subseteq \Omega_J$ and $\Delta \subseteq \Omega_K$ as above. Set $\kappa := \kappa(J,\Theta,K,\Delta)$, and $\Gamma := \Gal(\ksep/\kappa)$. Therefore the subsets $J$, $\Theta$, $K$, and $\Delta$ are $\Gamma$-invariant by definition. This implies that the $\Gamma$-action also preserves the subset $\doubleexp{J}{\What}{K}$ of $\What$ and induces an action on $\doubleexp{J,\Theta}{\What}{K,\Delta}$.  To simplify the notation we omit for schemes defined over a subfield of $\kappa$ the base change to $\kappa$. For instance we write simply $G$ instead of $G_{\kappa}$.

Consider the left diagonal action of $\Ghat$ on $\Par_{J,\Theta} \times \Par_{K,\Delta}$. The quotient stack
\begin{equation}\label{DefBruhatStack}
\Bcal_{J,\Theta,K,\Delta} := \Bcal_{J,\Theta,K,\Delta}(\Ghat) := [\Ghat\backslash (\Par_{J,\Theta} \times \Par_{K,\Delta})]
\end{equation}
is called the \emph{Bruhat stack of $\Ghat$ of type $(J,\Theta,K,\Delta)$}.

\begin{remark}\label{RemBruhat}
By Lemma~\ref{DoubleQuotient} one has $\Bcal_{J,\Theta,K,\Delta} = [\Phat_{J,\Theta}\backslash \Ghat/\Phat_{K,\Delta}]$.
\end{remark}

\begin{proposition}\label{TopBruhat}
The algebraic stack $\Bcal_{J,\Theta,K,\Delta}$ is of finite type and smooth of relative dimension $-\dim(L_J)$ over $\kappa$. Its underlying topological space of $\Bcal_{J,\Theta,K,\Delta}$ is homeomorphic to $\Gamma\backslash \doubleexp{J,\Theta}{\What}{K,\Delta}$. One has
\[
\pi_0(\Bcal_{J,\Theta,K,\Delta}) = \Gamma\backslash(\Omega_{J,\Theta}\backslash \Omega/\Omega_{K,\Delta}).
\]
\end{proposition}

Here we consider as usual the partially ordered set $\doubleexp{J,\Theta}{\What}{K,\Delta}$ as topological space (Remark~\ref{WConnComp}). As $\Gamma$ is also compatible with the Bruhat order, it acts by continuous automorphism on $\doubleexp{J,\Theta}{\What}{K,\Delta}$ and $\Gamma\backslash \doubleexp{J,\Theta}{\What}{K,\Delta}$ denotes the quotient space.

\begin{proof}
The first assertion is clear because $\Par_{J,\Theta}$ and $\Par_{K,\Delta}$ are smooth and of finite type over $\kappa$ and because $\dim(\Ghat) - (\dim(\Phat_{J,\Theta}) + \dim(\Phat_{K,\Delta})) = -\dim(L_J)$. The description of its underlying topological space follows from Proposition~\ref{NCBruhat} and the general description of the underlying topological space of quotient stacks given in~\cite{PWZ2}~1.2. The description of its set of connected componenents follows from the description of its underlying topological space and Lemma~\ref{ConnCompW}.
\end{proof}

Write $\what \in \doubleexp{J}{\What}{K}$ as $\what = w\omega$ with $\omega \in \Omega$ and $w \in \doubleexp{J}{W}{\omega(K)}$. We set
\[
J_{\what} := \omega(K) \cap w^{-1}Jw,
\]
and let $\what_{J,K}$ be the element of maximal length in $\leftexp{J_{\what}}{W_{\omega(K)}}$~\eqref{MaxLengthElt}. Then $w\what_{J,K}$ is the element of maximal length in $\leftexp{J}{W} \cap W_JwW_{\omega(K)}$ by a result of Howlett (\cite{PWZ1}~2.7 and~2.8). We define
\begin{equation}\label{DefEllJK}
\ell_{J,K}(\what) := \ell(w\what_{J,K}).
\end{equation}
By loc.~cit. one has
\begin{equation}\label{DescribeellJK}
\ell_{J,K}(\what) = \ell(w) + \ell(\what_{J,K}) = \ell(w) + \ell(w_{0,K}) - \ell(w_{0,J_{\what}}).
\end{equation}

\begin{remark}\label{GammaAndLength}
It is easy to check that left (resp.~right) multiplication of $\what$ with elements of $\Omega_J$ (resp.~of $\Omega_K$) does not change $\ell_{J,K}(\what)$. Moreover, the $\Gamma$-action preserves the length of elements and the sets $J$ and $K$. Therefore $\ell_{J,K}(\what) = \ell_{J,K}(\gamma(\what))$ for all $\gamma \in \Gamma$.
\end{remark}

For $x \in \Gamma\backslash\doubleexp{J,\Theta}{\What}{K,\Delta}$ we denote by $\Gscr_{x}$ the corresponding residual gerbe (in the sense of \cite{LM}~(11.1)) of $\Bcal_{J,\Theta,K,\Delta}$. This is the unique reduced locally closed algebraic substack of $\Bcal_{J,\Theta,K,\Delta}$ whose underlying topological space consists of the point $x$.


\begin{proposition}\label{CodimBruhat}
The residual gerbe $\Gscr_x$ is an algebraic stack smooth of relative dimension $\ell_{J,K}(x) - \dim(P_K)$ over $\Spec \kappa$.
\end{proposition}

\begin{proof}
One has
\[
\Gscr_{x} \otimes_{\kappa} \ksep = \coprod_{[\what] \in x \subset \doubleexp{J,\Theta}{\What}{K,\Delta}} \Gscr_{[\what]},
\]
where $\Gscr_{[\what]}$ is the residue gerbe of the point $[\what]$ of $\Bcal_{J,\Theta,K,\Delta}(G_{\ksep})$. Therefore we may assume that $k$ is separably closed and hence $\Gamma = 1$ and $x = [\what]$ for some $[\what] \in \doubleexp{J,\Theta}{\What}{K,\Delta}$. To simplify the notation we will for the rest of the proof confuse smooth algebraic groups $H$ and their $\ksep$-valued points. We write $\what = w\omega$ with $\omega \in \Omega$ and $w \in \doubleexp{J}{W}{\omega(K)}$.

The image of
\[
\Phat_{J,\Theta} \times \Phat_{K,\Delta} \to \Ghat, \qquad (p,q) \sends p\dot{\what} q
\]
is a locally closed subset. If we endow it with its reduced scheme structure it is a smooth subscheme $C_{J,\Theta,K,\Delta}(\what)$ of $\Ghat$ (being an orbit under the action of $\Phat_{J,\Theta} \times \Phat_{K,\Delta}$). As $\Phat_{J,\Theta}$ and $\Phat_{K,\Delta}$ are smooth, it is the fppf orbit of $\dot{\what}$ and $\Gscr_{[\what]} = [\Phat_{J,\Theta}\backslash C_{J,\Theta,K,\Delta}(\what)/\Phat_{K,\Delta}]$. This shows that $\Gscr_{[\what]}$ is smooth of relative dimension
\begin{equation}\label{EqDimGx}
\dim(\Gscr_{[\what]}/k) = \dim C_{J,\Theta,K,\Delta}(\what) - \dim \Phat_{J,\Theta} - \dim \Phat_{K,\Delta}
\end{equation}
over $k$. To calculate the right hand side we note that it does not change if we replace $\Phat_{J,\Theta}$ by $P_J$ and $\Phat_{K,\Delta}$ by $P_K$. Thus we assume from now on that $\Theta = \Delta = 1$ and hence $[\what] = \what$ for some $\what = w\omega \in \doubleexp{J}{\What}{K}$. Moreover, by right multiplication with $\omega^{-1}$ we may also assume that $\omega = 1$.

Recall that we chose a Borel $B$ of $G$ with $B \subset P_J \cap P_K$. Then
\[
P_J\dot{w} P_K = \bigcup_{v \in W_JwW_K} B\vdot B.
\]
But by a result of Howlett every $v \in W_JwW_K$ is uniquely expressible in the form $w_Jww_K$ with $w_J \in W_J$ and $w_K \in \leftexp{J_w}{W_K}$ (e.g., \cite{DDPW}, Proposition 4.18). Moreover $\ell(v) = \ell(w_J) + \ell(w) + \ell(w_K)$. Thus we obtain
\begin{align*}
&\bigcup_{v \in W_JwW_K} B\vdot B = \bigcup_{w_K \in \leftexp{J_w}{W_K}}\bigcup_{w_J \in W_J}B\dot{w}_J\dot{w}\dot{w}_KB \\
= &\bigcup_{w_K \in \leftexp{J_w}{W_K}} \bigcup_{w_J \in W_J}B\dot{w}_JB\dot{w}\dot{w}_KB = \bigcup_{w_K \in \leftexp{J_x}{W_K}} P_J\dot{w}\dot{w}_KB.
\end{align*}
The double coset $P_J\dot{w}\dot{w}_{J,K}B$ is open and dense in this union and hence $\dim C_{J,K}(w) = \dim P_J\dot{w}\dot{w}_{J,K}B$ and \eqref{EqDimGx} yields the desired description of $\dim(\Gscr_{\what}/k)$. 
\end{proof}

\begin{definition}\label{DefBruhatStrat}
Let $x \in \Gamma\backslash\doubleexp{J,\Theta}{\What}{K,\Delta}$. For every morphism of algebraic stacks $\beta\colon \Scal \to \Bcal_{J,\Theta,K,\Delta}$ we denote by $\leftexp{x}{\Scal(\beta)}$ the locally closed substack of $\Scal$ defined by the following two-cartesian diagram
\[\xymatrix{
\leftexp{x}{\Scal(\beta)} \ar[r] \ar[d] & \Gscr_{x} \ar[d] \\
\Scal \ar[r]^-{\beta} & \Bcal_{J,\Theta,K,\Delta}.
}\]
The family $(\leftexp{x}{\Scal(\beta)})_{x}$ is called the \emph{Bruhat stratification of $\Scal$ associated to $\beta$}.
\end{definition}

\begin{remark}\label{ConnComp}
For a morphism of algebraic stacks $\beta\colon \Scal \to \Bcal_{J,\Theta,K,\Delta}$ the inverse images of the connected components of $\Bcal_{J,\Theta,K,\Delta}$, which are para\-metrized by $\Gamma\backslash(\Omega_{J,\Theta}\backslash \Omega/\Omega_{K,\Delta})$ by Proposition~\ref{TopBruhat}, form open and closed algebraic substacks of $\Scal$.
\end{remark}


\subsection{Examples}

\begin{example}
The Bruhat stratification allows the following reinterpretation of the notion of relative position of two parabolic subgroups of $G$. Assume we are given $x \in \doubleexp{J}{W}{K}$ such that $\kappa(J) = \kappa(K) = \kappa(x) = k$. Let $S$ be a $k$-scheme and let $P$ and $Q$ be parabolic subgroups of $G_S$ of type $J$ and $K$, respectively. Then $(P,Q)$ corresponds to a morphism $S \to \Par_{J} \times_{\kappa} \Par_{K}$. Composing with the canonical quotient map to $\Bcal_{J,K}$ we obtain a morphism $S \to \Bcal_{J,K}$ and hence for $x \in \Gamma\backslash\doubleexp{J}{W}{K}$ a subscheme $S^{x}$ of $S$.
\end{example}

\begin{definition}\label{RelPos}
We say that $P$ and $Q$ are \emph{in relative position $x$} if $S^{x} = S$.
\end{definition}

\begin{example}\label{BruhatSymplectic}
Let $(V,\lrangle)$ be a symplectic space of dimension $2g$ over $k$ and let $G = \GSp(V,\lrangle)$ be the corresponding general symplectic group. Let $(W,I)$ be the Weyl group of $G$ together with its set of simple reflections. As $G$ is a split reductive group, the Galois action on $(W,I)$ is trivial. Let $L \subset V$ be a Lagrangian subspace (i.e., a totally isotropic subspace of dimension $g$) and let $P \subset G$ be its stabilizer. As $G(k')$ acts transitively on the set of Lagrangian subspaces of $V_{k'}$ for every extension $k'$ of $k$, $G/P$ is isomorphic to the $k$-scheme parametrizing Lagrangian subspaces of $V$. In other words for a $k$-scheme $S$ the set of $S$-valued points $(G/P)(S)$ is the set of totally isotropic locally direct summands of $V \otimes_k \Oscr_S$ of rank $g$. Let $J \subset I$ be the type of $P$ and set $K := J$. Then an easy calculation shows that the partially ordered set $\doubleexp{J}{W}{K}$ is isomorphic to the totally ordered set $\{0,\dots,g\}$ (see for instance~\cite{ViWd}~A.8).

The attached Bruhat stack $\Bcal_{J,J}(G)$ is the stack in groupoids fibered over the category of $k$-schemes whose fiber over a $k$-scheme $S$ is the category of triples $((\Fscr,\lrangle), \Lscr, \Lscr')$, where $(\Fscr,\lrangle)$ is a locally free $\Oscr_S$-module of rank $2g$ endowed with a symplectic pairing and where $\Lscr$ and $\Lscr'$ are Lagrangian subspaces of $\Fscr$. The morphisms $((\Fscr_1,\lrangle_1), \Lscr_1, \Lscr'_1) \to ((\Fscr_2,\lrangle_2), \Lscr_2, \Lscr'_2)$ in this category are symplectic similitudes $\psi\colon \Fscr_1 \to \Fscr_2$ such that $\psi(\Lscr_1) = \Lscr_2$ and $\psi(\Lscr'_1) = \Lscr'_2$.

If $S \to \Bcal_{J,J}(G)$ is a morphism corresponding to a triple $((\Fscr,\lrangle), \Lscr, \Lscr')$ as above, then for $x \in \{0,\dots,g\} = \doubleexp{J}{W}{K}$ the locally closed subscheme $\leftexp{x}{S}$ is defined by the following property. A morphism of schemes $f\colon T \to S$ factors through $\leftexp{x}{S}$ if and only if $f^*\Fscr/(f^*(\Lscr) + f^*(\Lscr'))$ is locally free of rank $g - x$.
\end{example}

\begin{remark}
In~\cite{Wd_AStrat} we will study the Bruhat stratification attached to the Hodge filtration and the conjugate filtration of the universal abelian scheme over the special fiber of Shimura varieties of PEL type of good reduction. Then Example~\ref{BruhatSymplectic} will show that in the Siegel case the Bruhat stratification is nothing but a scheme-theoretic version of the stratification by the $a$-number.
\end{remark}

%


\section{The Bruhat strata induced by $F$-zips with additional structures}

\subsection{$\Ghat$-zip functors}

In this section $\Ghat$ is a linear algebraic group with reductive identity component $G$ over a finite field $\FF_q$ with $q$ elements. We fix an algebraic closure $\FFbar_q$ of $\FF_q$. Otherwise we use the general notation introduced \ref{GenNot}. In particular $(W,I)$ denotes the Weyl group of $G$. We let $\bar\varphi$ be the automorphism of $\What$ induced by the geometric Frobenius $x \sends x^{1/q}$ which topologically generates $\Gal(\bar\FF_q/\FF_q)$. This is also the automorphism induced by the Frobenius isogeny $F\colon G \to G$ as explained in Section~\ref{IsogAut}. We denote by $\GhatRep$ the $\FF_q$-linear abelian tensor category of finite-dimensional rational representations of $\Ghat$ over $\FF_q$.

Let $S$ be an $\FF_q$-scheme. Recall from \cite{MoWd} that an \emph{$F$-zip over $S$} is a tuple $\underline{\Mscr} = (\Mscr,C^\bullet,D_\bullet,\varphi_\bullet)$ consisting of a locally free sheaf of $\Oscr_S$-modules of finite rank $\Mscr$ on~$S$, a descending filtration $C^\bullet$ and an ascending filtration $D_\bullet$ of~$\Mscr$, and an $\Oscr_S$-linear isomorphism $\varphi_i\colon (\gr_C^i\Mscr)^{(q)} \stackrel{\sim}{\to} \gr^D_i\Mscr$ for every $i\in\ZZ$, where $(\ )^{(q)}$ denotes the pullback by the Frobenius morphism $x\mapsto x^q$. In a natural way (see \cite{PWZ2}~Section~6) the $F$-zips over $S$ are the objects of an exact $\FF_q$-linear tensor category $\FZip(S)$.

More generally, as $F$-zips satisfy effective descent with respect to fpqc coverings, there is also the obvious generalization of an $F$-zip over an algebraic stack $\Scal$ defined over $\FF_q$. We obtain the exact $\FF_q$-linear tensor category $\FZip(\Scal)$.

In~\cite{PWZ2} ``$F$-zips with $\Ghat$-structures'' over an $\FF_q$-scheme $S$ are defined. These arise for instance from the De-Rham cohomology of smooth proper $S$-schemes $X$ whose Hodge spectral degenerates and commutes with arbitrary base change $S' \to S$ (see~\cite{PWZ2}~Section~9 for details and further examples). We now explain how every such ``$F$-zip with $\Ghat$-structure over $S$'' yields a Bruhat stratification of $S$.

First recall the precise definition of ``$F$-zip with $\Ghat$-structure''.

\begin{definition}
A \emph{$\Ghat$-zip functor over~$S$} is an exact $\FF_q$-linear tensor functor $\zfr\colon \GhatRep \to \FZip(S)$.
\end{definition}

Again, as also $\Ghat$-zip functors satisfy effective descent with respect to fpqc coverings (\cite{PWZ2}~Proposition~7.2), there is also the obvious generalization of a $\Ghat$-functor over an algebraic stack.

Let $k$ be a finite extension of $\FF_q$ and let $\mu$ be a cocharacter of $\Ghat_k$. To simplify later notation we assume that $\pi_0(\Ghat)_k$ is a constant group scheme (which always holds after passing to a finite extension of $k$). The homomorphism $\mu$ induces a grading on $V_k := V \otimes_{\FF_q} k$ for every representation $V$ of~$\Ghat$ and thus an $\FF_q$-linear tensor functor $\gamma_\mu$ from $\GhatRep$ to the category of graded $k$-vector spaces. On the other hand, any $\Ghat$-zip functor $\zfr$ over~$S$ induces an $\FF_q$-linear tensor functor from $\GhatRep$ to the category of graded locally free sheaves of $\Oscr_S$-modules on~$S$ which sends $V$ to $\gr^{\bullet}_C(\zfr(V))$. Then $\zfr$ is called \emph{of type $\mu$} if the graded fiber functors $\gr^{\bullet}_C \circ \zfr$ and $\gamma_{\mu}$ are fpqc-locally isomorphic.

\begin{remark}\label{GLnZip}
For classical groups $\Ghat$-zip functors of type $\mu$ are indeed equivalent to certain $F$-zips with additional structures. For instance, for $G = \GL_{n,\FF_q}$ evaluating a $G$-zip functor over $S$ in the standard representation yields by \cite{PWZ2}~8.1 an equivalence of the category of $G$-zip functors over $S$ and the category of $F$-zips $(\Mscr,C^\bullet,D_\bullet,\varphi_\bullet)$ of rank $n$ (i.e., $\rk_{\Oscr_S}(\Mscr) = n$) over $S$. Here such an $F$-zip corresponds to a $G$-zip functor of type $\mu$ if and only if $\mu$ is conjugate to $t \sends \diag(t^{r_1},\dots,t^{r_n})$ (with $r_i \in \ZZ$) and
\begin{equation}\label{GLnCond}
\rk_{\Oscr_S}(\gr^i_C\Mscr) = \#\set{j \in \{1,\dots,n\}}{r_j = i}.
\end{equation}
For further examples see the application using orthogonal groups below and \cite{PWZ2}~Section~8.
\end{remark}

We denote by $\GhatZipFunctor^{\mu}(S)$ the category of $\Ghat$-zip functors of type $\mu$ over $S$ (morphisms are morphisms of functors that are compatible with the tensor product). As $S$ varies, $\GZipFunctor^{\mu}$ is a category fibered in groupoids over the category of $k$-schemes. In fact, by results of \cite{PWZ1} and \cite{PWZ2} there is the following description of $\GZipFunctor^{\mu}$.

Let $P_\pm = L\ltimes U_\pm \subset G_k$ be the unique opposite parabolic subgroups with common Levi component $L$ and unipotent radicals~$U_\pm$, such that $\Lie U_+$ is the sum of the weight spaces of weights $>0$, and $\Lie U_-$ is the sum of the weight spaces of weights $<0$ in $\Lie G_k$ under ${\rm Ad}\circ\mu$. Set $\Lhat := \Cent_{\Ghat_k}(\mu)$ and $\Phat_{\pm} := \Lhat\ltimes U_\pm$. Then $\Phat_{\pm} \subseteq \Norm_{\Ghat}(P_{\pm})$ and $P_{\pm}$ is the identity component of $\Phat_{\pm}$. We set
\begin{equation}\label{DefineTheta}
\Theta := \pi_0(\Lhat) \cong \pi_0(\Phat_\pm) .
\end{equation}
As $\Ghat$ is defined over $\FF_q$, one has $\Ghat_k^{(q)} \cong \Ghat_k$. Via this isomorphism we can consider $\chi^{(q)}$ again as a cocharacter of~$G_k$, with associated subgroups $\Phat_\pm^{(q)} = \Lhat^{(q)}\ltimes U_\pm^{(q)}$.

The associated zip group (\cite{PWZ2}~Section~3.4) is the linear algebraic subgroup of $\Phat_+ \times_{\!k} \Phat_-^{(q)}$ defined as
\begin{equation}\label{ZipGroupDef}
E_{\Ghat,\mu} := 
\set{(\ell u_+,\ell^{(q)}u_-)}{\ell\in \Lhat,\ u_+\in U_+, u_-\in U_-^{(q)}}.
\end{equation}
It acts from the left hand side on $\Ghat_k$ by the formula 
\begin{equation}\label{ZipGroupActionDef}
(p_+,p_-)\cdot g \ := \ p_+ g p_-^{-1}.
\end{equation}
We explain now how to attach, for any $k$-scheme $S$, to $g \in \Ghat(S)$ a $\Ghat$-zip functor $\zfr(g)$ of type $\mu$ over $S$.

Let $\rho\colon \Ghat \to \GL(V)$ be a finite-dimensional representation of $\Ghat$ over $\FF_q$, let $V_S := V \otimes_{\FF_q} \Oscr_S$, and let $\rho(g) \in \GL(V_S)$ be the image of $g$ under $\rho(S)$. The cocharacter $\mu$ yields a grading on $V_k$. By base change to $S$ we obtain a grading $V_S = \bigoplus_{i\in\ZZ}V^i_S$ which induces a descending filtration $C^{\bullet}(V_S)$. As $V$ is defined over $\FF_q$, there is a natural identification $V_S^{(q)} \cong V_S$. Thus we may consider the decomposition $\bigoplus_{i\in\ZZ}\rho(g)((V^i_S)^{(q)})$ as another grading of $V_S$. Let $D_{\bullet}(V_S)$ be the induced ascending filtration. Finally let $\varphi_i$ for all $i \in \ZZ$ be the isomorphism
\begin{equation}\label{PhiFZip}
\varphi^{\rho(g)}_i\colon \gr^i_C(V_S)^{(q)} = (V^i_S)^{(q)} \liso \rho(g)((V^i_S)^{(q)}) = \gr^D_i(V_S),
\end{equation}
where the arrow in the middle is given by $\rho(g)$. In this way we obtain a functor $\zfr(g)$ by sending the representation $\rho$ to the $F$-zip $(V_S,C^{\bullet}(V_S), D_{\bullet}(V_S), \varphi^g_{\bullet})$ over $S$.

\begin{theorem}\label{GZipStack}
The above construction $g \sends \zfr(g)$ induces an isomorphism
\begin{equation}\label{DescribeGZipStack}
[E_{\Ghat,\mu}\backslash \Ghat_k] \liso \GhatZipFunctor^{\mu}
\end{equation}
of algebraic stacks. In particular, $\GhatZipFunctor^{\mu}$ is an algebraic stack, smooth of relative dimension $0$ over $k$.
\end{theorem}

\begin{proof}
This follows from combining the following results from \cite{PWZ2}: Proposition~3.11, Theorem~7.13, and Corollary~3.12.
\end{proof}

\begin{remark}\label{ConjStack}
Conjugation with $g \in \Ghat(k)$ yields an isomorphism
\[
\GhatZipFunctor^\mu \cong \GhatZipFunctor^{{\rm int}(g)\circ \mu}.
\]
\end{remark}


\subsection{Zip strata}

We now recall the description of the underlying topological space of $\GhatZipFunctor^{\mu}$. It is again the topological space attached to a finite partially ordered set, and we obtain for every $\Ghat$-zip over a scheme $S$ a stratification indexed by this partially ordered set. The strata are called zip strata.

Fix a cocharacter $\mu$ of $\Ghat_k$, $k$ finite extension of $\FF_q$, as above. We replace $\mu$ by ${\rm int}(g) \circ \mu$ for some $g \in G(k)$ such  that there exists a maximal torus $T$ of $G$ and a Borel subgroup $B$ of $G$ containing $T$ (both defined over $\FF_q$) such that $\mu$ factors through $T_k$ and is $B_k$-dominant. This implies $T_k \subset L$ and $B_k \subset P_+$. Let $J \subseteq I$ be the subset of simple reflections associated to $\mu$, i.e. \begin{equation}\label{DefineJ}
J := \set{i \in I}{\langle \mu, \alpha_i\rangle = 0},
\end{equation}
where $\alpha_i \in X^*(T)$ is the $B$-simple root corresponding to $i \in I$.

For the theory of $G$-zip functors we need only $\Theta$ as in~\eqref{DefineTheta}. But for the proof of Proposition~\ref{TopZip} below we allow the added flexibility that $\Theta \subseteq \Omega_J$ is an arbitrary subgroup scheme (automatically finite \'etale). We also set
\begin{equation}\label{DefineK}
K := \leftexp{w_0}{\bar\varphi(J)} = \bar\varphi(\leftexp{w_0}{J})
\end{equation}
(recall that $w_0$ denotes the greatest element in $W$). For the field of definitions of $J$ and $K$ we have $\kappa(J) = \kappa(K) \subseteq k$. We finally define
\begin{equation}\label{DefineDelta}
\Delta := \bar\varphi(\Theta).
\end{equation}
If $\Theta$ is defined as in~\eqref{DefineTheta}, then
\begin{equation}\label{DescribeDelta}
\Delta = \pi_0(\Lhat^{(q)}) = \pi_0(\Phat_{\pm}^{(q)}).
\end{equation}
Let $\Gamma = \Gal(\FFbar_q/k)$ be the Galois group of $k$. 

The parabolic subgroup $P_+$ is the unique parabolic subgroup of $G_k$ of type $J$, i.e. $P_+ = P_J$, where $P_J$ denotes the standard parabolic subgroup with respect to $(T,B)$ of type $J$ (Section~\ref{NonConnPSGP}). As $K$ is the type of the parabolic subgroup $P_-^{(q)}$ and as $P_-^{(q)}$ contains the Borel subgroup $\leftexp{w_0}{B^{(q)}} = \leftexp{w_0}{B}$, we find $P_-^{(q)} = \leftexp{\dot{w}_0}{P_K}$.
Thus if $\Theta$ is defined as in~\eqref{DefineTheta}, we find
\begin{equation}\label{DescribePhat}
\Phat_+ = \Phat_{J,\Theta}, \qquad \Phat^{(q)}_- = \leftexp{\dot{w}_0}{\Phat_{K,\Delta}}
\end{equation}
because conjugation with an element in $G(k)$ preserves each connected component of $G$.

For general $\Theta$ we still consider the algebraic zip datum (in the sense of \cite{PWZ1}~Definition~10.1)
\begin{equation}\label{DefineZipDatum}
\hat\Zcal := (\Ghat,\Phat_{J,\Theta}, \leftexp{\dot{w}_0}{\Phat_{K,\Delta}}, \varphi),
\end{equation}
where
\[
\varphi\colon \Lhat_{J,\Theta} \lto \Lhat_{\bar\varphi(J),\Delta} = \leftexp{\dot{w}_0}{\Lhat_{K,\Delta}}
\]
is the Frobenius isogeny. Let
\begin{equation}\label{DefineZipGroup}
E_{\Ghat,J,\Theta} = \set{(u\ell,v\varphi(\ell))}{u \in \Rcal_u(P_J), v \in \leftexp{\wdot_0}{\Rcal_u(P_K)}, \ell \in \Lhat_{J,\Theta}}
\end{equation}
be the associated zip group.

Let $x_0$ be the longest element in $\doubleexp{K}{W}{\bar\varphi(J)}$. Then
\begin{equation}\label{Shapex}
x_0 = w_{0,K}w_0 = w_0w_{0,\bar\varphi(J)} \in \doubleexp{K}{W}{\bar\varphi(J)},
\end{equation}
where $w_{0,K}$ and $w_{0,\bar\varphi(J)}$ denote the longest elements in $W_{K}$ and in $W_{\bar\varphi(J)}$, respectively. Let $g_0 \in \Norm_G(T)(k)$ be a representative of $x_0$ (this exists by Lemma~\ref{WdotRational} because $k$ is finite).

The automorphism of $\What$ 
\begin{equation}\label{DefinePsi}
\psi := \intaut(x_0) \circ \bar\varphi
\end{equation}
is $\Gamma$-equivariant, it induces an isomorphism of Coxeter system $(W_J,J) \iso (W_K,K)$, and it sends $\Omega_{J,\Theta}$ to $\Omega_{K,\Delta}$. The map
\begin{equation}\label{PsiAction}
(\omega,\what) \sends \omega\cdot_{\psi}\what := \omega\what\psi(\omega)^{-1}
\end{equation}
defines a left action of $\Omega_J$ on $\leftexp{J}{W}\Omega$ (\cite{PWZ1}~Lemma~10.4). As $x_0$ is fixed by $\Gamma$, this action is $\Gamma$-equivariant. We denote the set of $\Omega_{J,\Theta}$-orbits of $\leftexp{J}{W}\Omega$ under the action by
\begin{equation}\label{DefineXiJTheta}
\Xi_{J,\Theta} := \Omega_{J,\Theta}\mathop{\bs}_{\psi}\leftexp{J}{W}\Omega.
\end{equation}

The set $\Xi_{J,\Theta}$ is endowed with a partial order $\preceq$ as follows. For $\what, \what' \in \leftexp{J}{W}\Omega$ we define
\begin{equation}\label{CurlyOrder}
\what' \preceq \what \quad:\iff\quad \exists\,\vhat \in W_J\Omega_{J,\Theta}\colon\ \vhat\what'\psi(\vhat)^{-1} \leq \what.
\end{equation}
This defines a partial order on $\leftexp{J}{W}\Omega$ (\cite{PWZ1}~Theorem~10.9) preserved by the action of $\Gamma$. Moreover the above action of $\Omega_{J,\Theta}$ also preserves the partial order $\preceq$ (\cite{PWZ2}~Lemma~3.18). Thus we obtain an induced partial order (and hence a topology) on $\Xi_{J,\Theta}$.

Again we fix for $\omega \in \Omega$ a representative
\[
\dot\omega \in \Norm_{\Ghat}(B)(\FFbar_q) \cap \Norm_{\Ghat}(T)(\FFbar_q),
\]
and for $\what = w\omega \in \What$, $w \in W$, $\omega \in \Omega$, we set $\dot{\what} := \wdot\dot\omega \in \Norm_{\Ghat}(T)(\FFbar_q)$.

\begin{proposition}\label{TopZip}
Attaching to $\what \in \leftexp{J}{W}\Omega$ the $E_{\Ghat,J,\Theta}(\FFbar_q)$-orbit of $\dot\what g_0 \in \Ghat(\FFbar_q)$ yields a homeomorphism of the topological space associated to the partially ordered set $\Gamma\backslash \Xi_{J,\Theta}$ with the underlying topological space of the algebraic stack $[E_{\Ghat,J,\Theta}\bs \Ghat_k]$.
\end{proposition}

\begin{proof}
It is easy to check that $(T,\leftexp{g_0^{-1}\!\!}{B}, g_0)$ is a frame (\cite{PWZ1}~Definition~3.6) for the connected algebraic zip datum $\Zcal = (G,P_J, \leftexp{\dot{w}_0}{P_K}, \varphi)$. As $\leftexp{\dot{w}_0}{P_K} = \leftexp{g_0^{-1}\!\!}{P_K}$ conjugation by $g_0$ yields the frame $(T,B,g_0)$ of the connected algebraic zip datum $\Zcal' := (G,\leftexp{g_0}{P_J},P_K,\varphi)$, where again $\varphi$ denotes the Frobenius. Let $\hat\Zcal' := (\Ghat,\leftexp{g_0}{\Phat_{J,\Theta}},P_{K,\Delta},\varphi)$ and let $\Ehat'$ be the attached zip group. By \cite{PWZ1}~Proposition~7.3 the zip datum $\Zcal'$ is orbitally finite. Hence the $\Ehat'(\FFbar_q)$-orbit of $g_0\dot\what$ is the subset denoted by $\Ghat^{\what}$ in \cite{PWZ1}~(10.5). Therefore \cite{PWZ1}~Theorem~10.6 implies that $\what \sends g_0\dot\what$ induces a bijection from $\Gamma\backslash \Xi_{J,\Theta}$ to the set of $\Ehat'(\FFbar_q)$-orbits on $\Ghat(\FFbar_q)$. Moreover \cite{PWZ1}~Theorem~10.9 shows that this map is a homeomorphism. By composing with the inverse of the conjugation with $g_0$ we obtain the claim.
\end{proof}

\begin{corollary}\label{ZipConnComp}
Denote by $\Pi_{J,\Theta}$ the $\Omega_{J,\Theta}$-orbits on $\Omega$ under the left action $(\theta,\omega) \sends \theta\cdot_{\bar\varphi}\omega := \theta\omega\bar\varphi(\theta)^{-1}$. Then the homeomorphism in Proposition~\ref{TopZip} induces a bijection
\begin{equation}\label{EqZipConnComp}
\Gamma\bs\Pi_{J,\Theta} \lto \pi_0([E_{\Ghat,J,\Theta}\bs \Ghat_k]).
\end{equation}
\end{corollary}

\begin{proof}
For $\vhat = v\theta$ with $v \in W_J$, $\theta \in \Omega_{J,\Theta}$ and $\what = w\omega$ with $w \in \leftexp{J}{W}$, $\omega \in \Omega$ one has $\vhat\what\psi(\vhat)^{-1} = \tilde{w}\theta\omega\bar\varphi^{-1}(\theta)$ for some $\wtilde \in W$. Hence Proposition~\ref{TopZip} implies the claim using the definition of the partial order \eqref{CurlyOrder} which induces the topology of $[E_{\Ghat,J,\Theta}\bs \Ghat_k]$.
\end{proof}

By combining Proposition~\ref{TopZip} and Corollary~\ref{ZipConnComp} with Theorem~\ref{GZipStack} we obtain:

\begin{corollary}
The underlying topological space of $\GhatZipFunctor^{\mu}$ is homeomorphic to the topological space associated to the partially ordered set $\Gamma\backslash \Xi_{J,\Theta}$. The set of connected components of $\GhatZipFunctor^{\mu}$ can be identified with $\Gamma\bs\Pi_{J,\Theta}$.
\end{corollary}

Let $\zfr$ be a $\Ghat$-zip functor of type $\mu$ over a scheme $S$. Then $\zfr$ corresponds to a morphism $\zeta\colon S \to \GhatZipFunctor^{\mu}$ and we obtain a corresponding zip stratification of $S$ as follows.

For $\xi \in \Gamma\backslash\Xi_{J,\Theta}$ let $\Gscr_{\xi}$ be the corresponding residual gerbe of $\GhatZipFunctor^{\mu}$. For every morphism of algebraic stacks $\zeta\colon \Scal \to \GhatZipFunctor^{\mu}$ we denote by $\Scal^{\xi}(\zeta)$ (or $\Scal^{\xi}(\zfr)$ if $\zfr$ is a $\Ghat$-zip functor of type $\mu$ corresponding to $\zeta$) the locally closed substack of $\Scal$ defined by the following two-cartesian diagram
\begin{equation}\label{DiagramZipStrata}
\begin{aligned}\xymatrix{
\Scal^{\xi}(\zeta) \ar[r] \ar[d] & \Gscr_{\xi} \ar[d] \\
\Scal \ar[r]^-{\zeta} & \GhatZipFunctor^{\mu}.
}\end{aligned}
\end{equation}

\begin{definition}\label{DefZipStrat}
The family $(\Scal^{\xi}(\zeta))_{\xi\in \Gamma\backslash\Xi_{J,\Theta}}$ is called the \emph{zip stratification associated to $\zeta$} (or \emph{to $\zfr$}).
\end{definition}


\subsection{Zip strata and Bruhat strata}\label{ZipBruhatStrat}

We now construct a morphism from the algebraic stack of $\Ghat$-zips into the Bruhat stack. We keep the notation of the previous section, i.e., we fix $\mu$ and $k$ as above. From these data we obtain $J$ \eqref{DefineJ}. We fix $\Theta$. Again we are mainly interested in $\Theta$ as defined in \eqref{DefineTheta}. From $J$ and $\Theta$ we obtain $K$ \eqref{DefineK} and $\Delta$ \eqref{DefineDelta}. Finally we also fix a maximal torus $T$ and a Borel subgroup $B$ of $G$ and for every $w$ in the Weyl group of $G$ with respect to $T$ a representative $\wdot \in \Norm_{G}(T)(\FFbar_q)$ as above. For $w = x_0$ \eqref{Shapex} we denote a representative in $\Norm_G(T)(\FFbar_q)$ by $g_0$.


Define a morphism of $k$-schemes
\[
\bgtilde\colon \Ghat_k \to \Par_{J,\Theta} \times \Par_{K,\Delta}, \qquad g \sends (\Phat_{J,\Theta},g\dot{w}_0\Phat_{K,\Delta}).
\]
We obtain a composition of morphisms of algebraic stacks
\[
[E_{\Ghat,\mu}\bs \Ghat_k] \lto [(\Phat_{J,\Theta} \times \leftexp{\wdot_0}{\Phat_{K,\Delta}}) \bs \Ghat_k] \liso [\Ghat_k \bs (\Par_{J,\Theta} \times \Par_{K,\Delta})],
\]
where the first morphism is the canonical projection induced by the inclusion $E_{\Ghat,\mu} \subseteq \Phat_{J,\Theta} \times \leftexp{\wdot_0}{\Phat_{K,\Delta}}$ and where the second morphism is induced by $\bgtilde$ (it is an isomorphism by Lemma~\ref{DoubleQuotient}). Hence by Theorem~\ref{GZipStack} we obtain a morphism
\begin{equation}\label{DefineBeta}
\beta\colon \GhatZipFunctor^{\mu} \to \Bcal_{J,\Theta,K,\Delta}.
\end{equation}

\begin{proposition}\label{BetaSmooth}
The morphism $\beta$ is representable, of finite type and smooth of relative dimension $\dim L_J$.
\end{proposition}

\begin{proof}
This follows from the construction of $\beta$ and the fact that the quotient $(\Phat_{J,\Theta} \times \leftexp{\wdot_0}{\Phat_{K,\Delta}})/E_{\Ghat,\mu}$ has dimension $\dim L_J$.
\end{proof}

The morphism $\beta$ induces an open continuous map of the underlying topological spaces of these algebraic stacks which is given by Proposition~\ref{TopZip} and Proposition~\ref{TopBruhat} by a $\Gamma$-equivariant map
\begin{equation}\label{DefineBeta0}
\beta_0\colon \Xi_{J,\Theta}\to \doubleexp{J,\Theta}{\What}{K,\Delta}
\end{equation}
preserving the partial orders on these sets. We will now describe $\beta_0$.

Define a map
\begin{equation}\label{DefineBgtilde0}
\bgtilde_0\colon \What \to \doubleexp{J}{\What}{K}
\end{equation}
by sending $\what = w\omega \in \What$ with $w \in \leftexp{J}{W}$ and $\omega \in \Omega$ to $w^{\omega(K)}\omega$, where $w^{\omega(K)}$ is the element of minimal length in $W_JwW_{\omega(K)}$.

\begin{proposition}\label{ZipBruhatTop}
The restriction of $\bgtilde_0$ to $\doubleexp{J}{\What}{\emptyset}$ induces the map $\beta_0$ \eqref{DefineBeta0} that describes the underlying continuous map of $\beta$ \eqref{DefineBeta}.
\end{proposition}

\begin{proof}
On quotient stacks, $\beta$ is given by $[E_{\Ghat,J,\Theta}\backslash \Ghat_k] \to  [\Phat_{J,\Theta}\backslash \Ghat_k/\Phat_{K,\Delta}]$. The diagram
\[\xymatrix{
[E_{\Ghat,J,1}\backslash \Ghat_k] \ar[r] \ar[d] & [\Phat_{J,1}\backslash \Ghat_k/\Phat_{K,1}] \ar[d] \\
[E_{\Ghat,J,\Theta}\backslash \Ghat_k] \ar[r] & [\Phat_{J,\Theta}\backslash \Ghat_k/\Phat_{K,\Delta}]
}\]
is commutative, where the vertical morphisms are the canonical projections obtained by the inclusions $E_{\Ghat,J,1} \subseteq E_{\Ghat,J,\Theta}$, $\Phat_{J,1} \subseteq \Phat_{J,\Theta}$, and $\Phat_{K,1} \subseteq \Phat_{K,\Delta}$. Thus it suffices to prove the proposition in the case $\Theta = \Delta = 1$. 

By base change to a finite extension of $k$ we may assume that $\Gamma$ acts trivially on $\What$. By Remark~\ref{ConnComp}, all connected components of $[\Phat_{J,1}\backslash \Ghat_k/\Phat_{K,1}]$ are already geometric connected components. By Corollary~\ref{ZipConnComp}, the set of (geometric) connected components of $[E_{\Ghat,J,\Theta}\backslash \Ghat_k]$ is given by the set $\Pi_{J,\Theta}$ of $\Omega_{J,\Theta}$-orbits on $\Omega$ under the action $(\theta,\omega) \sends \theta\cdot_{\bar\varphi}\omega$, and the map induced by $\beta_0$ on the set of connected components is the canonical map $\Pi_{J,\Theta} \to \Omega_{J,\Theta}\bs\Omega/\Omega_{K,\Delta}$. Hence by restricting $\beta_0$ to connected components, we can reduce to the case that $\Ghat = G$ is connected.

Now note that $\wdot_0P_K = \wdot_0\wdot_{0,K}P_K = g_0^{-1}P_K$. Hence it follows from Proposition~\ref{TopZip} that the underlying continuous map of $\beta$ is induced by the map
\[
\leftexp{J}{W} \to (G/P_J)(\FFbar_q) \times (G/P_K)(\FFbar_q), \qquad w \sends (P_J(\FFbar_q),\wdot g_0g_0^{-1}P_K(\FFbar_q))
\]
Via the description of a double quotient in Lemma~\ref{DoubleQuotient} we thus see that $\beta_0$ is induced by
\[
\leftexp{J}{W} \to P_J(\FFbar_q)\bs G(\FFbar_q)/P_K(\FFbar_q), \qquad w \sends P_J(\FFbar_q)\wdot P_K(\FFbar_q).
\]
This shows the claim.
%
\end{proof}

The fibers of $\beta$ define the Bruhat stratification $(\leftexp{x}{\GhatZipFunctor^{\mu}(\beta)})_{x\in\doubleexp{J,\Theta}{\What}{K,\Delta}}$ on $\GhatZipFunctor^{\mu}$ (Definition~\ref{DefBruhatStrat})

For many classical groups $\Ghat$ a $\Ghat$-zip functor is the same as an $F$-zip with additional structures (\cite{PWZ2}~\S8). Moreover, often $\Par_{J,\Theta}$ and $\Par_{K,\Delta}$ classify flags with certain properties. Then $\beta$ is simply given by attaching to an $F$-zip $(\Mscr,C^\bullet,D_\bullet,\varphi_\bullet)$ with additional structure the underlying flags of $C^{\bullet}$ and $D_{\bullet}$. We make this more precise for $\GL_n$.

\begin{example}\label{GLnExample}
Let $\Ghat = G = \GL_{n,\FF_q}$ and let $\mu$ be the cocharacter $t \sends \diag(t^{r_1},\dots,t^{r_n})$ with $r_1 \geq \dots \geq r_n$. In particular we have $k = \FF_q$. Let $T$ be the diagonal torus and $B$ the Borel subgroup of upper triangular matrices. Then $\mu$ is dominant with respect to $B$. It induces a grading on $V = \FF_q^n$ and hence a descending filtration $C^{\bullet}$ and an ascending filtration $D_{\bullet}$ on $V$. The stabilizer of $C^{\bullet}(V)$ in $\GL_n$ is the parabolic subgroup $P_J$ where $J$ is defined as in~\eqref{DefineJ}. The stabilizer of $D_{\bullet}(V)$ is $\leftexp{w_0}{P_K}$, where $K = \leftexp{w_0}{J}$.

As explained in Remark~\ref{GLnZip}, evaluation in the standard representation of $\GL_n$ identifies the category of $\GL_n$-zip functors of type $\mu$ with the category $\FZip^{\mu}(S)$ of $F$-zips satisfying~\eqref{GLnCond}. Via this identification, the isomorphism \eqref{DescribeGZipStack} is induced by attaching to $g \in \GL_n(\Oscr_S)$ ($S$ some $\FF_q$-scheme) the $F$-zip $\underline{\Mscr}(g) = (V_S,C^{\bullet}(V)_S,g(D_{\bullet}(V)_S, \varphi_{\bullet}^g)$ with $\varphi_i^g$ as in \eqref{PhiFZip}.

For $G = \GL_n$ the Bruhat stack $\Bcal_{J,K}$ is isomorphic to the algebraic stack ${\tt Flag}_{J,K}$ classifying over an $\FF_q$-scheme $S$ triples $(\Mcal,\Fcal,\Gcal)$, where $\Mcal$ is a finite locally free $\Oscr_S$-module of rank $n$ and where $\Fscr$ (resp.~$\Gscr$) is a flag of $\Mscr$ whose stabilizer is a parabolic subgroup of $\GL_n(\Oscr_S)$ of type $J$ (resp.~of type $K$).

By the description of $\beta$ \eqref{DefineBeta} we hence obtain a commutative diagram
\[\xymatrix{
\text{$\GL_n$-{\tt ZipFun}}^{\mu} \ar[rr]^{\beta} \ar[d]_{\cong} & & \ar[d]^{\cong} \Bcal_{J,K}\\
\FZip^{\mu}(S) \ar[rr] & & {\tt Flag}_{J,K}
}\]
where the lower horizontal morphism sends an $F$-zip $(\Mscr,C^{\bullet},D_{\bullet},\varphi_{\bullet})$ to $(\Mscr,{\rm Fl}(C^{\bullet}), {\rm Fl}(D_{\bullet}))$, where ${\rm Fl}(\cdot)$ denotes the underlying flag of a filtration.
\end{example}

We now return to the case that $\Theta$ is defined by \eqref{DefineTheta}.

Given a $\Ghat$-zip functor of type $\mu$ over a $k$-scheme $S$ which corresponds to a morphism $\zeta\colon S \to \GhatZipFunctor^{\mu}$ amd hence a zip stratification $(S^{\xi})_{\xi\in\Gamma\bs\Xi_{J,\Theta}}$ (Definition~\ref{DefZipStrat}) of $S$. Moreover composition of $\zeta$ with $\beta$ yields a morphism $S \to \Bcal_{J,\Theta,K,\Delta}$ and hence a Bruhat stratification $(\leftexp{x}{S})_{x\in\Gamma\bs\doubleexp{J,\Theta}{\What}{K,\Delta}}$ (Definition~\ref{DefBruhatStrat}). Hence by definition, every Bruhat stratum is a union of zip strata. This union can be described as follows.

\begin{corollary}\label{DescribeFibers}
Let $x \in \Gamma\bs\doubleexp{J,\Theta}{\What}{K,\Delta}$ be the $\Gamma$-orbit of $\Omega_{J,\Theta}w\omega\Omega_{K,\Delta}$ with $\omega \in \Omega$ and $w \in \doubleexp{J}{W}{\omega(K)}$. Then the Bruhat stratum $\leftexp{x}{S}$ is the  union of the zip strata given by the $\Gamma$-orbits of $\Omega_{J,\Theta}\cdot_{\psi}wy\omega\delta$ where $y \in \leftexp{\omega(K) \cap w^{-1}Jw}{W_{\omega(K)}}$ and $\delta \in \Omega_{K,\Delta}$.
\end{corollary}

\begin{proof}
The automorphism $\psi$ \eqref{DefinePsi} induces an isomorphism $\Omega_{J,\Theta} \iso \Omega_{K,\Delta}$. Thus for every double coset $\Omega_{J,\Theta}\what\Omega_{K,\Delta}$ representatives for the action \eqref{PsiAction} are given by $\what\delta$ for $\delta \in \Omega_{K,\Delta}$. Moreover, the fiber in an element $w \in \doubleexp{J}{W}{\omega(K)}$ of the map $\leftexp{J}{W} \to \doubleexp{J}{W}{\omega(K)}$, $w \sends w^{\omega(K)}$ induced by $\bgtilde_0$ consists of the elements $wy$ with $y \in \leftexp{\omega(K) \cap w^{-1}Jw}{W_{\omega(K)}}$ by a corollary of a result of Howlett on Coxeter groups (\cite{PWZ1}~Prop.~2.8).
\end{proof}

\section{The example of the moduli space of K3-surfaces}

\subsection{Twisted orthogonal $F$-zips associated to surfaces}\label{FZipsSurface}

Let $S$ be an algebraic stack over $\FF_p$ where $p$ is an odd prime. Let $X$ be a surface over $S$. By this we mean a smooth proper morphism $f\colon X \to S$ representable by algebraic spaces and that has geometrically integral fibers of dimension $2$ (these are automatically schemes). We assume that $f$ satisfies the following two conditions.
\begin{simplelist}
\item[\textup{(D1)}]
The sheaves of $\Oscr_S$-modules $R^bf_*(\Omega^a_{X/S})$ are locally free of finite rank for all $a,b \geq 0$.
\item[\textup{(D2)}]
The Hodge-de Rham spectral sequence
\[
{}_HE^{ab}_1 = R^bf_*(\Omega_{X/S}^a) \implies H^{a+b}_{{\rm DR}}(X/S)
\]
degenerates at $E_1$.
\end{simplelist}
Then the formation of the Hodge-de Rham spectral sequence commutes with base change $S' \to S$, and $H^d_{\rm DR}(X/S)$ is locally free of finite rank for all $d \geq 0$. 

To simplify the exposition we assume that the locally constant Hodge numbers $h^{ab} := \rk_{\Oscr_S}(R^bf_*(\Omega^a_{X/S}))$ are in fact constant. Below we will focus on the case that $X$ is K3-surface over $S$, and then (D1) and (D2) are always satisfied (cf.~\cite{Del_K3}~Prop.~2.2).

The assumptions (D1) and (D2) imply that the de Rham cohomology $H^{\bullet}_{\rm DR}(X/S)$ carries the natural structure of an $F$-zip over $S$ (with $q = p$). As explained in \cite{PWZ2}~Section~9.2, the cup product yields a non-degenerate symmetric pairing
\begin{equation}\label{CupProd2}
\cup\colon H^2_{\rm DR}(X/S) \otimes H^2_{\rm DR}(X/S) \to \One(2),
\end{equation}
where $\One(d)$ denotes the Tate F-zip of weight $d$ (\cite{PWZ2}~Example~6.6). In other words, the triple $\Hline^2_{\rm DR}(X/S) := (H^2_{\rm DR}(X/S), \One(2), \cup)$ is a twisted orthogonal $F$-zip in the following sense.

\begin{definition}\label{DefTwistedOrth}
A \emph{twisted orthogonal $F$-zip over $S$} is a triple $(\Mscrline,\Lscrline,B)$ consisting of an $F$-zip $\Mscrline$ over $S$, an $F$-zips $\Lscrline$ of rank $1$ over $S$ and a homomorphism of $F$-zips $B\colon \Sym^2(\Mscrline) \to \Lscrline$, whose underlying symmetric pairing is perfect.
\end{definition}

The type of the $F$-zip $H^2_{\rm DR}(X/S)$ is $\nline = (n_i)_{i\in\ZZ}$ with $n_0 = n_2 = 1$ (because of our assumption that $f$ has geometrically integral fibers), $n_1 = h^{11}$, and $n_i = 0$ for $i \ne 0,1,2$.

Note that all these assertions may be checked fpqc-locally on $S$ and thus follow from the analogue results in the case that $S$ is a scheme.


\subsection{Classification of some twisted orthogonal $F$-zips}

The case of surfaces leads us to study the following general situation. Let $q$ be an odd power of a prime. We let $\GG_m$ (resp.~$\mmu_r$) be the multiplicative group (resp.~the group scheme of $r$-th roots of unity) over $\FF_q$.

We now study a general twisted orthogonal $F$-zip $(\Mscrline,\Lscrline,B)$ over an $\FF_q$-scheme $S$, where $\Mscrline$ has type $\nline = (n_i)_{i\in\ZZ}$ with $n_0 = n_2 = 1$, $h := n_1 > 0$, and $n_i = 0$ for $i \ne 0,1,2$, and where $\Lscrline$ has type $2$. Let $n = h+2 = \rk_{\Oscr_S}(\Mscr)$, let $b$ be a non-degenerate split symmetric bilinear form on $V = \FF_q^n$, and let $\GO(V,b)$ be the group of orthogonal similitudes of $(V,b)$. Thus we have an exact sequence of algebraic groups
\[
1 \to {\rm O}(V,b) \lto \GO(V,b) \ltoover{\eta} \GG_m \to 1, 
\]
where $\eta$ denotes the multiplier homomorphism.

Fix $\nline$ as above. Let $\mu$ be a cocharacter of $\GO(V,b)$ (unique up to conjugation) whose weights on the standard representation $V$ of $\GO(V,b)$ are $i$ with multiplicity $n_i$ for alle $i$. By \cite{PWZ2}~8.6 and Theorem~\ref{GZipStack} the following algebraic stacks are equivalent.
\begin{equivlist}
\item
The stack of twisted orthogonal $F$-zip $\Mscr$ of type $\nline$.
\item
The stack of $\GO(V,b)$-zip functors of type $\mu$.
\item
The quotient stack $[E_{\GO(V,b),\mu}\bs \GO(V,b)]$.
\end{equivlist}
Hence the twisted orthogonal $F$-zip $(\Mscrline,\Lscrline,B)$ yields a zip stratification and a Bruhat stratification on $S$. We describe these strata in more detail. 

\subsubsection*{Flags and orthogonal spaces}
We start by reviewing some standard definitions. Let $S$ be a scheme and let $\Mscr$ be a finite locally free $\Oscr_S$-module. A \emph{flag in $\Mscr$} is a set $\Fcal$ of local direct summands of the $\Oscr_S$-module $\Mscr$ which is totally ordered by inclusion, which contains $0$ and $\Mscr$, and such that $\Escr(s) \subsetneq \Escr'(s)$ for all $s \in S$ whenever $\Escr \subsetneq \Escr'$ are two elements of $\Fcal$. The \emph{type of $\Fcal$} is the set of locally constant functions
\[
\set{\rk_{\Oscr_S}(\Escr)}{\Escr \in \Fcal, \Escr \ne 0,\Mscr}.
\]
If $\Escr$ is a local direct summand of $\Mscr$, we denote by $\Escr^{\bot} \subseteq \Mscr\vdual$ its orthogonal in the dual $\Mscr\vdual$.

\begin{definition}\label{DefOrthogonalSpace}
Let $S$ be any scheme over $\ZZ[1/2]$. A triple $(\Mscr,\Lscr,b)$, where $\Mscr$ is a finite locally free $\Oscr_S$-module, $\Lscr$ is an invertible $\Oscr_S$-module, and where $b\colon \Sym^2(\Mscr) \epi \Lscr$ is a perfect symmetric pairing, is called a \emph{space of orthogonal similitudes}. Its \emph{rank} is the rank of the $\Oscr$-module $\Mscr$.

A \emph{flag of $(\Mscr,\Lscr,b)$} is a flag $\Fcal$ of $\Mscr$, such that for all $\Escr \in \Fcal$ there exists $\Escr' \in \Fcal$ such that $b$ (considered as an isomorphism $\Mscr \iso \Mscr\vdual \otimes \Lscr$) induces an isomorphism $\Escr \iso (\Escr')^{\bot} \otimes \Lscr$.
\end{definition}

Again, all these notions satisfy effective descent for fpqc-coverings and in particular generalize immediately to the case that $S$ is an algebraic stack.

There is the obvious notion of an isomorphism of spaces of orthogonal similitudes over a fixed $\ZZ[1/2]$-scheme $S$. Let ${\tt SOS}_n(S)$ be the category of spaces of orthogonal similitudes of rank $n$ over $S$ where all morphisms are isomorphisms. Letting $S$ vary over all $\FF_q$-schemes and with the obvious notion of pull backs for scheme morphisms $S' \to S$ we obtain a category ${\tt SOS}_n$ fibered in groupoids over the category of $\FF_q$-schemes. 

\begin{lemma}\label{SOSClasify}
${\tt SOS}_n$ is the classifying algebraic stack $[\GO(V,b)\bs \Spec(\FF_q)]$ of $\GO(V,b)$.
\end{lemma}

\begin{proof}
As descent data of spaces of orthogonal similitudes and their morphisms are clearly effective for fpqc coverings, ${\tt SOS}_n$ is a stack. Then the lemma follows from the fact that any two spaces of orthogonal similitudes of the same rank on a $\ZZ[1/2]$-scheme $S$ are locally for the \'etale topology isomorphic (\cite{Kn_Quad}).
\end{proof}

We will now describe the groups $\GO(V,b)$ and their Weyl groups. Recall that $n = \dim_{\FF_q}(V) \geq 3$.

\subsubsection*{Case: $n$ odd}
If $n$ is odd, then there is an isomorphism
\[
\SO(V,b) \times \GG_m \liso \GO(V,b), \qquad (g,\lambda) \sends \lambda\id_V \circ g.
\]
Hence $\GO(V,b)$ is a connected split reductive group with root system of Dynkin type $B_{m}$ with $m := (n-1)/2$ (with the convention $B_1 := A_1$). Its Weyl group can be identified with
\[
W := \sett{\sigma \in S_n}{$\sigma(i) + \sigma(n+1-i) = n+1$ for all $i = 1,\dots,n$}.
\]
Then every $w \in W$ is already uniquely determined by $(w(i))_{1\leq i\leq m}$ and $w(m+1) = m+1$. The longest element $w_0$ is given by the permutation $i \sends n+1-i$ and hence $w_0$ is central. The set of simple reflections is given by $I = \{s_1,\dots,s_m\}$, where
\[
s_i = (i,i+1)(n-i,n-i+1), \quad i = 1,\dots,m-1, \qquad s_m = (m,m+2).
\]
Here $(i,j)$ denotes the transposition of $i$ and $j$. As $\GO(V,b)$ is split, the automorphism $\bar\varphi$ of $(W,I)$ is trivial. The cocharacter $\mu$ yields by \eqref{DefineJ} and \eqref{DefineK} the sets
\[
J = K = \{s_2,\dots,s_m\}.
\]

\begin{lemma}\label{DescribeJWOdd}
The Bruhat order and the order $\preceq$ \eqref{CurlyOrder} on $\leftexp{J}{W}$ coincide and are total orders. The length function induces an isomorphism of totally ordered sets
\begin{equation}\label{EqJW}
\leftexp{J}{W} \liso \{0,\dots,n-2\}, \qquad w \sends \ell(w).
\end{equation}
Under this isomorphism the subset $\doubleexp{J}{W}{K}$ of $\leftexp{J}{W}$ corresponds to the set $\{0,1,n-2\}$.
\end{lemma}

\begin{proof}
It is easy to check that
\begin{align*}
\{0,\dots,n-2\} &\to \leftexp{J}{W}, \\
d &\sends
\begin{cases}
s_1\dots s_{d},&0 \leq d \leq m;\\
s_1\dots s_ms_{m-1}\dots s_{2m-d},&m+1 \leq d \leq 2m-1;
\end{cases}
\end{align*}
defines an inverse of \eqref{EqJW}. This shows that \eqref{EqJW} is an isomorphism of totally ordered sets, where $\leftexp{J}{W}$ is endowed with the Bruhat order. The order $\preceq$ is a refinement of the Bruhat order which still satisfies the relation $w \preceq w' \implies \ell(w) \leq \ell(w')$ (by \cite{PWZ1}~Theorem~5.11). Thus it has to agree with the Bruhat order. The last assertion is clear.
\end{proof}

Finally note that $\Xi_{J,1} = \leftexp{J}{W}$ \eqref{DefineXiJTheta} because $\GO(V,b)$ is connected. Therefore the underlying topological space of the stack of  twisted orthogonal $F$-zips of type $\nline$ (with $\nline$ as above and $n$ odd) is the topological space attached to the totally ordered set $\{0,\dots,n-2\}$.

\subsubsection*{Case: $n$ even}
If $n$ is even, $\GO_n$ has two connected components and the homomorphism $\GO_n \to \mmu_2$, $g \sends \eta(g)/\det(g)^{n/2}$ induces an isomorphism $\pi_0(\GO_n) \iso \mmu_2$. The identity component $G$ of $\GO_n$ is a split reductive group of Dynkin type $D_m$ with $m := n/2$ (with the convention $D_2 := A_1 + A_1$ and $D_3 := A_3$). The case $m = 2$ would require some extra notation, thus we assume for simplicity $m \geq 3$. The Weyl groups of $G$ and of $\GO_n$ can be identified with
\begin{align*}
\What &:= \set{\sigma \in S_n}{\sigma(i) + \sigma(n+1-i) = n+1},\\
W &:= \set{\sigma \in \What}{\text{$\#\set{1\leq i \leq m}{\sigma(i) > m}$ is an even number}}.\end{align*}
Again, every $\what \in \What$ is determined by $(\what(i))_{1\leq i\leq m}$. We set
\[
\omega := (m,m+1),
\]
where again $(i,j)$ denotes the transposition of $i$ and $j$. Then
\[
\Omega = \{\id, \omega\}.
\]
The set of simple reflections is given by $I = \{s_1,\dots,s_m\}$ with
\begin{align*}
s_i &:= (i,i+1)(n-i,n-i+1), \qquad i=1,\dots,m-1,\\
s_m &:= (m-1,m+1)(m,m+2),
\end{align*}
Note that $s_{m-1}s_m = s_ms_{m-1}$. Let $\what_0 \in \What$ be the permutation $i \sends n+1-i$. Then the longest element $w_0$ of $W$ is given by
\[
w_0 = \begin{cases}
\what_0,&\text{if $m$ is even};\\
\what_0\omega,&\text{if $m$ is odd}.
\end{cases}
\]
As above, the automorphism $\bar\varphi$ of $\What$ is trivial and the cocharacter $\mu$ yields the sets
\[
J = K = \{s_2,\dots,s_m\}.
\]
We have $\omega(s_i) = s_i$ for $1 \leq i \leq m-2$ and $\omega(s_{m-1}) = s_m$. In particular $\omega(K) = K$ and hence
\[
\doubleexp{J}{\What}{K} = \doubleexp{J}{W}{K} \cup \doubleexp{J}{W}{K}\omega.
\]
Moreover, $\Cent_{\GO_n}(\mu) \cong \GO_{n-2} \times \GG_m$. Hence if we define $\Theta$ and $\Delta$ as in \eqref{DefineTheta} and \eqref{DefineDelta}, then we get
\[
\Theta = \Delta = \pi_0(\GO_{n-2}) \cong \mmu_2
\]
and hence $\Omega_{J,\Theta} = \Omega_{K,\Theta} = \Omega$. It follows that the canonical map $\doubleexp{J}{W}{K} \to \doubleexp{J,\Theta}{\What}{K,\Delta}$ is an isomorphism of partially ordered sets which we use to identify these partially ordered sets.

We have a direct sum of partially ordered sets
\[
\leftexp{J}{W}\Omega = \leftexp{J}{W} \cup \leftexp{J}{W}\omega
\]
and the action \eqref{PsiAction} of $\Omega_{J,\Theta} = \Omega$ preserves this decomposition.

\begin{lemma}\label{DescribeJWEven}
The Bruhat order and the order $\preceq$ \eqref{CurlyOrder} on $\leftexp{J}{W}$ coincide. It is given by
\begin{align*}
t_0 := \id &\leq t_1 := s_1 \\
&\leq \cdots \\
&\leq t_{m-2} := s_1\cdots s_{m-2} \\
&\leq t_{m-1} := s_1\cdots s_{m-2}s_{m-1}, t'_{m-1} := s_1\cdots s_{m-2}s_m \\
&\leq t_{m} := s_1 \cdots s_m \\
&\leq t_{m+1} := s_1\cdots s_{m}s_{m-2} \\
&\leq \cdots \\
&\leq t_{2m-2} := s_1\cdots s_ms_{m-2}\cdots s_1,
\end{align*}
and each element is written as a reduced product of simple reflections (in particular $\ell(t_i) = i$). Moreover $\doubleexp{J}{W}{K} = \{\id, t_1, t_{2m-2}\}$. The action of $\omega$ on $\leftexp{J}{W}$ fixes the $t_i$ for $i \ne m-1$, and it interchanges $t_{m-1}$ and $t'_{m-1}$.
\end{lemma}

\begin{proof}
For the Bruhat order this is easy to check. The order $\preceq$ is a refinement of the Bruhat order which also satisfies the relations $w \preceq w' \implies \ell(w) \leq \ell(w')$ and
\[
w \preceq w', \ell(w) = \ell(w') \implies w = w'
\]
(again by \cite{PWZ1}~Theorem~5.11). Thus it has to agree with the Bruhat order. The remaining assertions are clear.
\end{proof}

Hence we see that the partially ordered set $\Xi_{J,\Theta}$ \eqref{DefineXiJTheta}, which describes the underlying topological space of the stack of twisted orthogonal $F$-zips of type $\nline$ ($\nline$ as above, $n$ even), is the direct sum (as partially ordered set) of the two totally ordered sets $\Omega\bs\leftexp{J}{W}$ and $\Omega\bs\leftexp{J}{W}\omega$. For both of these totally ordered sets the length function is an isomorphism of ordered sets with $\{0,\dots,n-2\}$.

\subsubsection*{Description of the Bruhat stack and the zip stack}

In both cases (i.e., $n$ odd and $n$ even) we see that there are precisely three Bruhat strata parametrized by
\[
\doubleexp{J}{W}{K} = \doubleexp{\Theta,J}{\What}{K,\Delta} = \{\id,s_1,w_1\},
\]
where $w_1$ is the element of maximal length in $\doubleexp{J}{W}{K}$. Therefore Proposition~\ref{TopBruhat} implies that the Bruhat stack $\Bcal_{J,K}(\GO(V,b))$ has three points $\id$, $s_1$, and $w_1$, where $\id$ is a specialization of $s_1$, and $s_1$ is a specialization of $w_1$. The corresponding residual gerbes classify the following data.

\begin{remark}\label{DescribeBruhatK3}
For every $\FF_q$-scheme $S$ the group $\GO(V,b)(S)$ acts \'etale locally on $S$ transitively on the set orthogonal flags in $(V,b)_S$ of type $\{1,n-1\}$. This, together with Lemma~\ref{SOSClasify}, implies that the Bruhat stack $\Bcal_{J,K}(\GO(V,b))$ classifies over an $\FF_q$-scheme $S$ tuples $(\Mscr,b,\Lscr,\Fcal,\Gcal)$, where $(\Mscr,b,\Lscr)$ is a space of orthogonal similitudes of rank $n$ over $S$ and where $\Fcal = \{\Fscr_1,\Fscr_{n-1}\}$ and $\Gcal = \{\Gscr_1,\Gscr_{n-1}\}$ are two flags of $(\Mscr,b,\Lscr)$ that both have the type $\{1,n-1\}$. Moreover the Bruhat stratification is given as follows:
\begin{assertionlist}
\item
$\leftexp{\id}{S}$ is the closed subscheme of $S$ where $\Fcal = \Gcal$. More precisely, it is the unique closed subscheme of $S$ such that a morphism $f\colon T \to S$ of schemes factors through $\leftexp{\id}{S}$ if and only if $f^*\Fcal = f^*\Gcal$.
\item
$\leftexp{s_1}{S}$ is the locally closed subscheme of $S$ where $\Fscr_1 \subseteq \Gscr_{n-1}$ (or, equivalently, $\Gscr_1 \subseteq \Fscr_{n-1}$) and where $\Fscr_1 + \Gscr_1$ is a direct summand of $\Mscr$ of rank $2$.
\item
$\leftexp{w_1}{S}$ is the open subscheme of $S$, where $\Fscr_1 + \Gscr_{n-1} = \Mscr$ (or, equivalently, $\Gscr_1 + \Fscr_{n-1} = \Mscr$).
\end{assertionlist}
\end{remark}

Now we can use Proposition~\ref{TopZip} to deduce that the underlying topological space of the stack of twisted orthogonal $F$-zips of type $\nline$ is the topological space attached to the partially ordered set
\begin{equation}\label{DefineXiN}
\Xi_n := \Xi_{J,\Theta},
\end{equation}
and the description of $\Xi_n$ obtained above yields the following result.

\begin{proposition}\label{DescribeTwistedOrth}
If $n$ is odd (resp.~if $n$ is even) the stack of twisted orthogonal $F$-zips of type $\nline$ has one (resp.~two) connected components which are also geometrically connected. The underlying topological space of each connected component is homeomorphic to the topological space associated to the totally ordered set $\{0,\dots,n-2\}$.
\end{proposition}


\subsection{Zip strata and Bruhat strata for K3-surfaces}

Let $S$ be an algebraic stack over $\FF_p$ ($p$ an odd prime) and let $X \to S$ be a K3-surface over $S$. By this we mean that $X$ is a surface over $S$ as defined in the beginning of Subsection~\ref{FZipsSurface} such that each geometric fiber $X_{\sbar} \to \Spec(\kappa(\sbar))$ is a K3-surface in the usual sense (i.e., $\Omega^2_{X_{\sbar}/\sbar}$ is trivial and $H^1(X_{\sbar},\Oscr_{X_{\sbar}}) = 0$).

Then the conditions (D1) and (D2) above are satisfied and we obtain a twisted orthogonal $F$-zip
\[
\Mcal := \Mcal(X) := \Hline^2_{\rm DR}(X/S)
\]
as explained in Subsection~\ref{FZipsSurface}. Its rank is $22$.

If we endow $X$ with a $p$-principal polarization, we may attach also a primitive version of $\Mcal$. For this recall that a \emph{polarization on $X$} is a class $\lambda \in \Pic_{X/S}(X)$ (where $\Pic_{X/S}$ denotes the relative Picard scheme) such that for every geometric point $\sbar$ of $S$ the fiber $\lambda_{\sbar}$ is the class of an ample line bundle. Then for $s \in S$ the self intersection number $d(s) := (\lambda_{\sbar}.\lambda_{\sbar})$ ($\sbar$ some geometric point over $s$) is always an even integer and the function $d\colon S \to 2\ZZ$ is locally constant on $S$. A polarization $\lambda$ is called \emph{$p$-principal} if its degree is prime to $p$.

If $\lambda$ is a polarization on $X$, its first Chern class $c_1(\lambda)$ is a global section of $H^2_{\rm DR}(X/S)$ with
\begin{equation}\label{CupChern}
c_1(\lambda) \cup c_1(\lambda) = d(\lambda) \in \Gamma(S,\Oscr_S).
\end{equation}
Form now on we assume that $\lambda$ is $p$-principal. Then \eqref{CupChern} shows that $H^2_{\rm DR}(X/S) = H_{\rm prim}(X,\lambda) \oplus \langle c_1(\lambda)\rangle$, where $\langle c_1(\lambda)\rangle$ is the $\Oscr_S$-module generated by $c_1(\lambda)$ and where $H_{\rm prim}(X,\lambda)$ is defined as the orthogonal complement of $\langle c_1(\lambda)\rangle$ in $H^2_{\rm DR}(X/S)$. Moreover, it follows from results of Ogus (\cite{Og_SSK3}), that $\langle c_1(\lambda)\rangle \cong \One(1)$ and that the cup product induces on $\langle c_1(\lambda)\rangle$ the canonical pairing $\One(1) \otimes \One(1) \to \One(2)$. In particular we see that the cup product induces the structure of a twisted orthogonal $F$-zip on $H_{\rm prim}(X,\lambda)$. We denote it by $\Mcal' = \Mcal'(X,\lambda)$. Its rank is $21$ and its type is $\nline'$ with $n'_0 = n'_2 = 1$, $n'_1 = 19$ and $n'_i = 0$ for $i \ne 0,1,2$.

Let $(X,\lambda)$ be a $p$-principally polarized K3-surface over $S$ of constant degree $2d$. The twisted orthogonal $F$-zips $\Mcal(X)$ and $\Mcal'(X,\lambda)$ induce zip strata \[
S = \bigcup (S^\xi)_{\xi \in \Xi_{22}} = \bigcup S^{\prime \xi}_{\xi \in \Xi_{21}}
\]
and Bruhat strata
\[
S = \leftexp{\id}{S} \cup \leftexp{s_1}{S} \cup \leftexp{w_1}{S} = \leftexp{\id}{S'} \cup \leftexp{s_1}{S'} \cup \leftexp{w_1}{S'}.
\]
The description of the Bruhat strata in Remark~\ref{DescribeBruhatK3} shows that these two Bruhat stratifications coincide.

The two connected components $\Xi_{22,1} := \Omega \bs \leftexp{J}{W}$ and $\Xi_{22,2} := \Omega \bs \leftexp{J}{W}\omega$ of $\Xi_{22}$ yields a decomposition of $S$ into two open and closed substacks $S_1$ and $S_2$. It is easy to see that this decomposition is given by the discriminant of orthogonal $F$-zips explained in \cite{Wd_Goett}~(4.3). As the discriminant is the class of $-2d$ in $\FF_p^\times/(\FF_p^\times)^2$ (by \cite{Wd_Goett}~Lemma~5.3), we see that one of this open and closed substacks is empty. More precisely, $S_2 = \emptyset$ if and only if $-2d$ is a square in $\FF_p^\times$.

Finally, the comparison between zip strata and Bruhat strata (Corollary~\ref{DescribeFibers}) shows that
\begin{align*}
\leftexp{\id}{S} &= S^{\id} \cup S^{\omega} = S^{\prime\id}, \\
\leftexp{s_1}{S} &= \bigcup_{\substack{\xi\in\Xi_{22} \\ 1 \leq \ell(\xi) \leq 19}}S^{\xi} = \bigcup_{\substack{\xi\in\Xi_{21} \\ 1 \leq \ell(\xi) \leq 18}}S^{\prime\xi}, \\
\leftexp{w_1}{S} &= S^{t_{20}} \cup S^{t_{20}\omega} = S^{\prime w_1}.
\end{align*}

\bibliography{references}

\end{document}